\documentclass[11pt,a4paper]{scrartcl}

\usepackage[centertags]{amsmath} 
\usepackage{color} 
\usepackage{amsfonts}
\usepackage{amssymb}
\usepackage{amsthm}
\usepackage{epsfig} 
\usepackage{rotating}
\usepackage{psfrag}
\usepackage{caption}
\usepackage{booktabs}
\usepackage{enumitem}
\usepackage{pdflscape}

\newcommand{\N}{{\mathbb N}}

\newcommand{\R}{{\mathbb R}}

\newcommand{\XSet}{{\mathcal X}}
\newcommand{\USet}{{\mathcal U}}
\newcommand{\ZSet}{{\mathcal Z}}
\newcommand{\FSet}{{\mathcal F}}
\newcommand{\ESet}{{\mathcal E}}
\newcommand{\TSet}{{\mathcal T}}
\newcommand{\PSet}{{\mathcal P}}
\newcommand{\SSet}{{\mathcal S}}
\newcommand{\KSet}{{\mathcal K}}
\newcommand{\DSet}{{\mathcal D}}

\newcommand{\xb}{\boldsymbol{x}}
\newcommand{\vb}{\boldsymbol{v}}

\newcommand{\Ab}{\boldsymbol{A}}
\newcommand{\Bb}{\boldsymbol{B}}
\newcommand{\Cb}{\boldsymbol{C}}
\newcommand{\Kb}{\boldsymbol{K}}
\newcommand{\Sb}{\boldsymbol{S}}
\newcommand{\Qb}{\boldsymbol{Q}}
\newcommand{\Pb}{\boldsymbol{P}}

\newcommand{\IFun}{{\mathcal I}}

\newcommand{\proj}{\mathrm{proj}}
\newcommand{\rank}{\mathrm{rank}}
\newcommand{\vol}{\mathrm{vol}}

\newcommand{\ul}[1]{\underline{#1}} 
\newcommand{\ol}[1]{\overline{#1}}

\newtheorem{thm}{Theorem}
\newtheorem{cor}[thm]{Corollary}
\newtheorem{lem}[thm]{Lemma}
\newtheorem{assum}{Assumption}
\newtheorem{prop}[thm]{Proposition}

\newtheorem{rem}{Remark}

\begin{document}

\begin{center}
{\LARGE
\textbf{
\textsf{On closed-loop dynamics of ADMM-based MPC}}
}

\renewcommand{\thefootnote}{$\dagger$} 

\vspace{5mm}
{
Moritz Schulze Darup\footnotemark[1] and Gerrit Book\footnotemark[1]
}
\vspace{2mm}

  \footnotetext[1]{M. Schulze Darup and G. Book are with 
the Automatic Control Group,  Department of Electrical Engineering and Information Technology,
        Universit\"at Paderborn, Warburger Stra\ss e 100,
33098 Paderborn, Germany.
        E-mail: {\texttt{moritz.schulzedarup@rub.de}}.}
\end{center}

\paragraph{Abstract.}

This paper studies the closed-loop dynamics of linear systems under approximate model predictive control (MPC). More precisely, we consider MPC implementations based on a finite number of ADMM iterations per time-step. 
We first show that the closed-loop dynamics can be described based on a nonlinear augmented model. We then characterize an invariant set around the augmented origin, where the dynamics become linear. 
Finally, we investigate the performance of the approximate MPC for various choices of the ADMM parameters based on a comprehensive numerical benchmark.


\section{Introduction}

Model predictive control (MPC) is a popular optimization-based control strategy. Typically, an
optimal control problem (OCP) is solved in every time step to evaluate the control
action to be applied. The objective function of this OCP specifies the performance metric, while the constraints encode a
model of the system as well as constraints on states and inputs. For a convex quadratic performance metric, a linear model, and polytopic constraints, the
resulting OCP is a convex quadratic program (QP). Several methods for solving QPs arising in control exist. Examples include interior-point methods \cite{Wang2010}, active-set procedures \cite{Ferrau2014}, multiparametric programming \cite{Bemporad2002}, and proximal algorithms such as projected gradient schemes \cite{Patrinos2014,Giselsson2015}
and the alternating direction method of multipliers (ADMM) \cite{ODonoghue2013,Jerez2014}. The various solvers have in common that they iteratively approach the optimal solution. It is usually assumed that the number of iterations is high enough to approximate the optimum sufficiently well.
This assumption can be hard to realize for MPC implementations tailored for resource-constrained embedded platforms, networked systems, or very high sampling rates.
For those applications, termination of the optimization after a small number of iterations can be required even if the iterates have not yet converged to the optimum.

At first sight, ``incomplete'' optimization seems to be doomed to fail. However, in the framework of optimization-based control, a fixed number of iterations per time step can be sufficient since additional iterations follow at future sampling instances. In the resulting setup, optimization-iterates are, to some extent, coupled to sampling times and thus called real-time iterations. MPC based on real-time iterations has been realized using various optimization schemes. Newton-type single and multiple shooting solvers are considered in \cite{Diehl2005_Single} and \cite{Diehl2005}. A projected gradient scheme and real-time ADMM have recently been discussed in~\cite{VanParys2018} and \cite{SchulzeDarup2019_ECC_ADMM}, respectively.
All of  these works focus on the special case of a single optimization iteration per time step. 
Moreover, state and input constraints are often not considered. In fact, constraints are neglected in \cite{Diehl2005_Single,Diehl2005} (at least for the theoretical statements) and only input constraints are included in \cite{VanParys2018}.

In this paper, motivated by the promising results in \cite{SchulzeDarup2019_ECC_ADMM}, we study MPC based on real-time ADMM for linear systems with state and input constraints. However, in contrast to \cite{SchulzeDarup2019_ECC_ADMM} and previous works on real-time iterations, we allow multiple iterations per time step.
Nevertheless, the number of iterations $M$ is a priori fixed for every sampling instant and, in particular, independent of the current system state. 
As a consequence, the control law is explicitly defined based on the $M$ ADMM iterations.
 The purpose of this work is to analyze how the closed-loop system dynamics change with the parameters of the ADMM scheme. To this end, inspired by \cite{VanParys2018} and \cite{SchulzeDarup2019_ECC_ADMM}, we show that an augmented state space model allows to describe the behavior of the controlled system. The subsequent analysis of the augmented system is twofold. First, a theoretical analysis will reveal some characteristics of the augmented system. 
  Second, a numerical benchmark will indicate that MPC based on real-time ADMM is competitive (compared to standard MPC) if a suitable parametrization is used. For example, in one analyzed scenario with $M=10$ iterations per time step (that is specified in line $8$ of Tab.~\ref{tab:Benchmark}), $486$ out of $500$ initial states (that are feasible for the original MPC) are steered to the origin without violating the constraints and with a performance decrease of only $0.03\%$ (compared to the original MPC).

The paper is organized as follows.  Some frequently used notation is introduced in the remainder of this section. 
In Section~\ref{sec:Background}, we summarize basic results on MPC for linear systems and corresponding implementations using ADMM.
In Section~\ref{sec:augmentedModel}, we introduce the real-time ADMM scheme and derive the augmented model for the closed-loop dynamics. We show in Section~\ref{sec:LinearAroundOrigin} that the dynamics become linear around the augmented origin.
We further investigate the linear regime in Sections~\ref{sec:positiveInvariance} and~\ref{sec:costToGo} by specifying an invariant set and the cost-to-go around the origin, respectively. In Sections~\ref{sec:designParameters} and \ref{sec:benchmark}, we discuss different choices of the ADMM parameters and analyze their impact based on a numerical benchmark. Finally, conclusions and an outlook are given in Section~\ref{sec:Conclusions}.

\subsection*{Notation}
\label{sec:Notation}

The sets of natural numbers (including~$0$) and real numbers are denoted by $\N$ and $\R$, respectively. The identity matrix in $\R^{n \times n}$ is called $I_n$.
With $0_{m \times n}$, we denote the zero matrix in $\R^{m \times n}$.
For the zero vector in $\R^m$, we write $0_{m}$ instead of $0_{m \times 1}$.  
For vectors $x \in \XSet \subset \R^n$ and $z \in \ZSet \subset \R^q$, 
 we occasionally write
$$
\begin{pmatrix}
x \\
z 
\end{pmatrix} \in \XSet \times \ZSet,
$$
instead of $(x,z) \in \XSet \times \ZSet$, i.e., we omit the splitting of concatenated vectors into vector pairs. Vector-valued inequalities such as $z \leq \overline{z}$ with $\overline{z} \in \R^q$ are understood element-wise. The indicator function $\IFun_\ZSet$ of some set $\ZSet$ is defined as
$$
\IFun_\ZSet(z) := \left\{ \begin{array}{ll}
0 & \quad \text{if} \quad z \in \ZSet, \\
\infty & \quad \text{if} \quad z \notin \ZSet.
\end{array}\right.
$$

\section{Background on ADMM-based MPC for linear systems}
\label{sec:Background}

We here consider linear discrete-time systems
\begin{equation}
\label{eq:linearSys}
x(k+1)=A x(k) + B u(k), \qquad x(0):=x_0
\end{equation}
with state and input constraints of the form
\begin{subequations}
\label{eq:constraintsXU}
\begin{align}
x(k) \in \XSet &:= \{ x \in \R^{n} \,|\, \ul{x} \leq x \leq \ol{x} \} \quad \text{and} \\
 u(k) \in \USet &:= \{ u \in \R^{m} \,|\, \ul{u} \leq u \leq \ol{u} \}.
\end{align}
\end{subequations}
The box-constraints are characterized by bounds $\ul{x}, \ol{x} \in \R^n$ and $\ul{u}, \ol{u} \in \R^m$ 
satisfying $\ul{x}<0_n<\ol{x}$ and $\ul{u} < 0_m < \ol{u}$.
Now, standard MPC is based on solving the OCP 
\begin{subequations}
\label{eq:OCP}
\begin{align}
\label{eq:VPerformance}
\!V_N(x)  :=  \min_{\substack{\hat{x}(0),\dots,\hat{x}(N),\\\hat{u}(0),\dots,\hat{u}(N-1)}} & \!\!\varphi(\hat{x}(N))  + \sum_{k=0}^{N-1} \ell(\hat{x}(k),\hat{u}(k) )   \span \span   \\
\label{eq:x0Constraint}
\text{s.t.} \qquad  \hat{x}(0) & = x, && \\
\label{eq:xPlusConstraint}
 \hat{x}(k+1)&=A \hat{x}(k) + B \hat{u}(k) && \forall k  \in \{0,...,N-1\}, \\
\label{eq:XConstraint}
 \hat{u}(k) &  \in \USet && \forall k  \in \{0,...,N-1\}, \\
\label{eq:UConstraint}
 \hat{x}(k) & \in \XSet  && \forall k  \in \{1,...,N\}
\end{align}
\end{subequations}
in every time step for the current state $x=x(k)$. The objective function thereby consists of quadratic cost functions  
\begin{equation}
\label{eq:stageCost}
\varphi(\hat{x}):=\hat{x}^\top P \hat{x} \qquad \text{and} \qquad \ell(\hat{x},\hat{u}):=\hat{x}^\top Q \hat{x}+ \hat{u}^\top R \hat{u},
\end{equation}
where the weighting matrices $Q$ and $R$ are design parameters and where $P$ is chosen as the solution of the discrete-time algebraic Riccati equation (DARE)
\begin{equation}
\label{eq:DARE}
A^\top ( P- P\,B\,(R+B^\top P\,B )^{-1} B^\top P)\,A - P + Q = 0.
\end{equation}
The control action in every time step then is 
\begin{equation}
\label{eq:optimalInput}
u(k)=\hat{u}^\ast(0),
\end{equation}
 i.e., the first element of the optimal control sequence for~\eqref{eq:OCP}.
For completeness, we make the following standard assumptions. 

\begin{assum}
The pair $(A,B)$ is stabilizable, that $R$ is positive definite, and that $Q$ can be written as $L^\top L$ with $(A,L)$ being detectable.
\end{assum}

We additionally note that no terminal set is considered in~\eqref{eq:OCP}. Recursive feasibility and convergence to the origin may thus not hold for every initially feasible state
\begin{equation}
\label{eq:setFN}
x \in \FSet_N := \{ x \in \XSet \,|\, \eqref{eq:OCP}\,\,\text{is feasible}\}.
\end{equation}
 We stress, however, that recursive feasibility and convergence guarantees can be obtained for almost every $x \in \FSet_N$ by suitably choosing the horizon length $N$\,(see, e.g., 
 \cite{Grieder2004b}, \cite[Thm.~13]{Boccia2014}, or \cite[Thm.~3]{SchulzeDarup2016_ECC_MPC}).

It is easy to see that the OCP~\eqref{eq:OCP} is a QP parametrized by the current state~$x$. As stated in the introduction, we here use ADMM (see, e.g., \cite{Boyd2011}) for its approximate solution.
 To prepare the application of ADMM, we rewrite~\eqref{eq:OCP} as the QP 
\begin{subequations}
\label{eq:QP}
\begin{align}
\label{eq:QPCost}
V_N(x)  = \min_{z \in \ZSet} \,\,\frac{1}{2} z^\top H z & +x^\top Q x\\
\label{eq:QPEqualityConstraint}
\text{s.t.} \quad  \quad  G z &= F x
\end{align}
\end{subequations}
with the decision variables
\begin{equation}
\label{eq:zDefinition}
z^\top:= \begin{pmatrix}
\begin{pmatrix}
\hat{u}(0) \\
\hat{x}(1)
\end{pmatrix}^\top & \dots & \begin{pmatrix}
\hat{u}(N-1) \\
\hat{x}(N)
\end{pmatrix}^\top 
\end{pmatrix},
\end{equation}
the constraint set
\begin{equation}
\label{eq:setZ}
\ZSet := \{ z \in \R^{q} \,|\, \ul{z} \leq z \leq \ol{z} \},
\end{equation}
and suitable matrices $F \in \R^{p \times n}$, $G \in \R^{p \times q}$, and $H \in \R^{q \times q}$, where $p:=N n $ and $q:=p+Nm$.
We first note that the constraint~\eqref{eq:x0Constraint} and the associated variable $\hat{x}(0)$ have been eliminated in~\eqref{eq:QP}. 
We further note that the specific order of $\hat{x}(k)$ and $\hat{u}(k)$ in~\eqref{eq:zDefinition} facilitates some  mathematical expressions in the subsequent sections.
We finally note that the bounds  $\ul{z},\ol{z} \in \R^q$ and the matrix $H$ are uniquely determined by the constraints~\eqref{eq:constraintsXU}, the cost functions \eqref{eq:VPerformance} and \eqref{eq:stageCost}, and definition \eqref{eq:zDefinition}.
In contrast, $G$ and $F$ are not unique. To simplify reproducibility of our results, we use 
\begin{equation}
\label{eq:definitionGF}
G:= \begin{pmatrix}
-B &\,\,\,\, I_n &\,\,\,\, 0_{n \times m} &\,\,\,\, 0_{n \times n} & & \,\,\,\, 0_{n \times n}\\
0_{n \times m} &\,\, -A &\,\, -B &\,\,  I_n \\
\vdots & &\,\,\,\, \ddots &\,\,\,\, \ddots &\,\,\,\, \ddots \\
0_{n \times m}& & &\,\,\,\, -A &\,\,\,\, -B &\,\,\,\,  I_n
\end{pmatrix} \quad \text{and} \quad
F:=\begin{pmatrix}
A \\
0_{(p-n)\times n}
\end{pmatrix} 
\end{equation}
throughout the paper. Now, by introducing the set
$$
\ESet(x):= \{ z \in \R^{q} \,|\,G z = F x \}, 
$$
it is straightforward to show that the optimizers of~\eqref{eq:QP} and 
\begin{subequations}
\label{eq:copyOCP}
\begin{align}
\label{eq:copyObjective}
 \min_{y,z}\, \frac{1}{2} y^\top \! H y &+ \IFun_{\ESet(x)}(y)+\IFun_{\ZSet}(z)+ \frac{\rho}{2} \| y - z\|_2^2,\\
\label{eq:copyConstraint}
\text{s.t.} \quad  \quad  y &=z,
\end{align}
\end{subequations}
are equivalent (see \cite[Eqs.~(9)-(10)]{Jerez2014}) for any positive $\rho \in \R$. Due to~\eqref{eq:copyConstraint}, the decision variable $y$ acts as a copy of~$z$. Hence, $y$ and $z$ are interchangeable in~\eqref{eq:copyObjective}. However, the specific choice in \eqref{eq:copyObjective} turns out to be useful \cite{Jerez2014}.
We next investigate the dual problem to~\eqref{eq:copyOCP}, which is given by 
\begin{equation}
\label{eq:dualOCP}
\max_{\mu} \,\,  \left(\inf_{y,z} L_\rho(y,z,\mu) \right),  
\end{equation}
where the (augmented) Lagrangian with the Lagrange multipliers $\mu$ reads
$$
L_\rho(y,z,\mu):=\frac{1}{2} y^\top H y + \IFun_{\ESet(x)}(y)+\IFun_{\ZSet}(z) + \frac{\rho}{2} \| y-z \|_2^2 + \mu^\top (y-z).
$$
ADMM solves~\eqref{eq:dualOCP} by repeatedly carrying out the iterations
\begin{subequations}
\begin{align}
\label{eq:yIter}
y^{(j+1)}&:= \arg \min_{y}  L_\rho\big(y,z^{(j)},\mu^{(j)}\big), \\
z^{(j+1)}&:= \arg \min_{z}   L_\rho\big(y^{(j+1)},z,\mu^{(j)}\big),  \quad \text{and}\\
\label{eq:muIter}
\mu^{(j+1)} &:= \mu^{(j)} + \rho\, \big( y^{(j+1)} - z^{(j+1)} \big)
\end{align}
\end{subequations}
(cf.~\cite[Sect.~3.1]{Boyd2011}).
Since $z^{(j)}$ and $\mu^{(j)}$ are constant in~\eqref{eq:yIter}, we obviously have
\begin{align*}
y^{(j+1)}  = \arg \min_{y} \, \frac{1}{2} y^\top (H &+ \rho I_p) y + \big(\mu^{(j)} - \rho z^{(j)}\big)^\top \!y \\
\text{s.t.} \quad  \quad  G y &= F x.
\end{align*}
The solution to this equality-constrained QP results from
\begin{equation}
\label{eq:yPlus1Equation}
\begin{pmatrix}
H + \rho I_q & G^\top \\
G & 0_{p \times p} 
\end{pmatrix} \begin{pmatrix}
y^{(j+1)} \\
\ast
\end{pmatrix} = \begin{pmatrix}
\rho z^{(j)} - \mu^{(j)} \\
F x
\end{pmatrix}.
\end{equation}
Thus, precomputing the matrix
\begin{equation}
\label{eq:EBlocks}
E=\begin{pmatrix}
E_{11} & E_{12} \\
E_{12}^\top & E_{22}
\end{pmatrix} := \begin{pmatrix}
H + \rho I_q & G^\top \\
G & 0_{p \times p} 
\end{pmatrix}^{-1}
\end{equation}
allows to evaluate 
\begin{equation}
\label{eq:yIterResult}
y^{(j+1)} = E_{11} \big( \rho z^{(j)} - \mu^{(j)} \big) + E_{12} F x. 
\end{equation}
According to \cite[Sect.~III.B]{Jerez2014}, we further have
\begin{equation}
\label{eq:zIterResult}
z^{(j+1)}=\proj_\ZSet \bigg(  y^{(j+1)} + \frac{1}{\rho} \mu^{(j)} \bigg),
\end{equation}
i.e., $z^{(j+1)}$ results from projecting $y^{(j+1)} + \rho^{-1} \mu^{(j)}$ onto the set $\ZSet$.
In summary, by substituting $ y^{(j+1)}$ from~\eqref{eq:yIterResult} into~\eqref{eq:zIterResult} and~\eqref{eq:muIter}, we obtain the two iterations
\begin{subequations}
\label{eq:iterationsFinal}
\begin{align}
\label{eq:zIterFinal}
z^{(j+1)}&:= \proj_\ZSet \bigg(  E_{11} \big( \rho z^{(j)} - \mu^{(j)} \big)+ E_{12} F x + \frac{1}{\rho} \mu^{(j)} \!\bigg)\qquad \text{and} \\
\label{eq:muIterFinal}
\mu^{(j+1)}&:= \mu^{(j)} + \rho \left( E_{11} \big( \rho z^{(j)} - \mu^{(j)} \big) +E_{12} F x  - z^{(j+1)} \right)
\end{align}
\end{subequations}
that are independent of the copy $y$ introduced in~\eqref{eq:copyOCP}.
The iterations~\eqref{eq:iterationsFinal} form the basis for our subsequent analysis of  ADMM-based MPC.
For their derivation, we followed  the procedure in \cite{Jerez2014} that considered  the ``uncondensed'' QP~\eqref{eq:QP} with equality constraints. We stress that ADMM-based MPC can be implemented differently. For example, an implementation for the ``condensed'' QP without equality constraints is discussed in \cite{Ghadimi2015}.
For the analysis to be presented, the uncondensed form is more intuitive. 
However, in contrast to our approach, the method in \cite{Ghadimi2015} allows to efficiently incorporate terminal constraints. In this context, we note that ADMM can only be efficiently applied if the underlying iterations are easy to evaluate, where 
the crucial step is usually a projection similar to~\eqref{eq:zIterFinal}.
Here, the projection onto $\ZSet$ can be efficiently evaluated due to~\eqref{eq:setZ}.
In fact, for such box-constraints, we easily compute
\begin{equation}
\label{eq:projZz}
\left(\proj_{\ZSet}(z) \right)_i = \left\{ \begin{array}{ll}
\ol{z}_i & \quad \text{if} \quad \ol{z}_i < z_i, \\
z _i &  \quad\text{if} \quad \ul{z}_i \leq z_i \leq \ol{z}_i, \\
\ul{z}_i &  \quad \text{if} \quad z_i < \ul{z}_i. 
\end{array}\right.
\end{equation}
We note the consideration of a terminal set in~\eqref{eq:OCP} (i.e., $\hat{x}(N) \in \TSet$ instead of $\hat{x}(N) \in \XSet$) will, in general, result in a non-trivial set $\ZSet$ and, hence, in a non-trivial projection. However, we could consider terminal sets that allow for efficient projections such as, e.g., box-shaped or ellipsoidal sets $\TSet$.
We further note that implementing \eqref{eq:zIterResult} can be more efficient than \eqref{eq:zIterFinal} if a sparse factorization of the matrix in~\eqref{eq:yPlus1Equation} is used to compute $y^{(j+1)}$ instead of the dense matrix $E$.

\section{An augmented model for the closed-loop dynamics}
\label{sec:augmentedModel}

Generally, many iterations~\eqref{eq:iterationsFinal} are required to solve~\eqref{eq:QP} (resp.~\eqref{eq:OCP}) with a certain accuracy. The number of required iterations varies, among other things, with the current state. 
Here, we fix the number of iterations to $M \in \N$ for the whole runtime of the controller and consequently independent of the current state. In contrast to existing works, we do not investigate the accuracy of the resulting ADMM scheme for a specific state but we study the dynamics of the corresponding closed-loop system in general.
To this end, we initially provide a more precise description of the controlled system. In every time step $k$, 
we compute $z^{(M)}(k)$ and $\mu^{(M)}(k)$ according to \eqref{eq:iterationsFinal} based on $x(k)$, $z^{(0)}(k)$, and $\mu^{(0)}(k)$.
Inspired by~\eqref{eq:optimalInput}, the input~$u(k)$ is then chosen as the first element of the input sequence contained in $ z^{(M)} (k)$, i.e.,
\begin{equation}
\label{eq:uADMM}
u(k)=C_u z^{(M)}(k) \qquad \text{with}  \qquad C_u:=\begin{pmatrix}
I_m\,\,\, & 0_{m\times (q-m)}
\end{pmatrix}. 
\end{equation}
Obviously, the resulting input depends on the initializations
$z^{(0)}(k)$ and $\mu^{(0)}(k)$ of the iterations \eqref{eq:iterationsFinal}. In principle, we can freely choose these initializations in every time step.
However, it turns out to be useful to reuse data from  previous time steps. This approach is called warm-start and it is well-established in optimization-based control. More precisely, we choose the initial values $z^{(0)}(k+1)$ and $\mu^{(0)}(k+1)$ at step $k+1$ based on the final iterates $z^{(M)}(k)$ and $\mu^{(M)}(k)$ from step $k$. Moreover, for simplicity,  we restrict ourselves to linear updates of the form
\begin{subequations}
\label{eq:zMuUpdates}
\begin{align}
z^{(0)}(k+1) &:= D_z z^{(M)} (k)  \qquad \text{and} \\
 \mu^{(0)}(k+1) &:= D_\mu \mu^{(M)} (k), 
\end{align}
\end{subequations}
where we note that similar updates have been considered in \cite[Sect.~2.2]{Diehl2005_Single}. Suitable choices for $D_z$ ad $D_\mu$ will be  discussed later in Section~\ref{sec:designParameters}. Here, we focus on 
the structural dynamics of the controlled system.
To this end, we first note that the open-loop dynamics~\eqref{eq:linearSys}, the $M$ iterations~\eqref{eq:iterationsFinal}, the input~\eqref{eq:uADMM}, and the updates~\eqref{eq:zMuUpdates} determine the closed-loop behavior.
More formally, we show in the following that the controlled system can be described using the augmented state
\begin{equation}
\label{eq:xxDefintion}
\xb  := \begin{pmatrix}
x \\
z^{(0)} \\
\mu^{(0)} 
\end{pmatrix} \in \R^{r}
\end{equation}
with $r:=n+2q$. In fact, according to \eqref{eq:linearSys}, \eqref{eq:uADMM}, and \eqref{eq:zMuUpdates} the closed-loop dynamics are captured by the augmented system 
\begin{equation}
\label{eq:augmentedSys}
\xb(k+1)=\Ab \xb(k)+\Bb \vb ( \xb(k)) \,,\qquad \xb(0):=\xb_0.
\end{equation}
with the augmented system matrices
$$
\Ab := \begin{pmatrix}
A & 0_{n\times q} & 0_{n\times q} \\
0_{q \times n} & 0_{q\times q} &  0_{q\times q} \\
0_{q \times n} &  0_{q\times q} &  0_{q\times q}
\end{pmatrix} \qquad \text{and} \qquad \Bb:= \begin{pmatrix}
B C_u & 0_{n \times q}\\
D_z & 0_{q \times q} \\
0_{q \times q} & D_\mu
\end{pmatrix},
$$
and the augmented control law $\vb:\R^r \rightarrow \R^{2q}$ with
\begin{equation}
\label{eq:vControlLaw}
\vb(\xb):=\begin{pmatrix}
z^{(M)} \\
\mu^{(M)} 
\end{pmatrix}.
\end{equation}
According to the following proposition,  $\vb(\xb)$ is continuous and  piecewise affine in~$\xb$.

\begin{prop}
Let $\vb$ be defined as in~\eqref{eq:vControlLaw} with $z^{(M)}$ and $\mu^{(M)}$ resulting from $M \in \N$ iterations \eqref{eq:iterationsFinal}. Then,  $\vb$ is continuous and piecewise affine in $\xb$.
\end{prop}

\begin{proof}
We prove the claim by showing that the iterates $z^{(j)}$ and $\mu^{(j)}$ are continuous and piecewise affine in $\xb$ not only for $j=M$ but for every $j \in \{0, \dots, M\}$.  We prove the latter statement by induction.
For $j=0$, continuity and piecewise affinity hold by construction since
$$
\begin{pmatrix}
z^{(0)} \\
\mu^{(0)} 
\end{pmatrix} = \begin{pmatrix}
0_{2q \times n} & I_{2q}
\end{pmatrix} \xb,
$$
i.e., $z^{(0)}$ and $\mu^{(0)}$ are linear in $\xb$.
It remains to prove that $z^{(j)}$ and $\mu^{(j)}$ being continuous and piecewise affine implies continuity and piecewise affinity of $z^{(j+1)}$ and $\mu^{(j+1)}$.
To this end, we first note that the iterations~\eqref{eq:iterationsFinal} can be rewritten as
\begin{equation}
\label{eq:zjMujPlus1}
z^{(j+1)} = \proj_{\ZSet}\left(\Kb^{(1)} \begin{pmatrix}
x \\
z^{(j)}\\
\mu^{(j)} 
\end{pmatrix}\right) \quad \text{and} \quad \mu^{(j+1)} =  \rho \left(\Kb^{(1)} \begin{pmatrix}
x \\
z^{(j)}\\
\mu^{(j)}
\end{pmatrix} - z^{(j+1)}\right)
\end{equation}
with
\begin{equation}
\label{eq:K1}
\Kb^{(1)} := \begin{pmatrix}
E_{12} F &\quad \rho E_{11} &\quad \frac{1}{\rho} I_q - E_{11}
\end{pmatrix}.
\end{equation}
Clearly, 
$$
\zeta^{(j)}:=\Kb^{(1)} \begin{pmatrix}
x \\
z^{(j)}\\
\mu^{(j)} 
\end{pmatrix}
$$
is continuous and piecewise affine in $\xb$ if these properties hold for $z^{(j)}$ and $\mu^{(j)}$. Moreover $\proj_{\ZSet}(\zeta)$ is continuous and piecewise affine in $\zeta$ as apparent from~\eqref{eq:projZz}.
Since the composition of two continuous and piecewise affine functions results in a continuous and piecewise affine function, $z^{(j+1)}$ as in~\eqref{eq:zjMujPlus1} is indeed continuous and piecewise affine in $\xb$. As a consequence, $\mu^{(j+1)}=\rho (\zeta^{(j+1)}-z^{(j+1)})$ is continuous and piecewise affine in $\xb$ since it results from the addition of two continuous and piecewise affine functions.
\end{proof}

Obviously, the augmented system \eqref{eq:augmentedSys} inherits continuity and piecewise affinity from $\vb(\xb)$. This trivial result is summarized in the following corollary for reference. Moreover, we point out a connection between the augmented systems in \eqref{eq:augmentedSys} and \cite[Eq.~(21)]{SchulzeDarup2019_ECC_ADMM} in Remark~\ref{rem:augmentedSystems}. 

\begin{cor}
Let $\vb$ be defined as in~\eqref{eq:vControlLaw} with $z^{(M)}$ and $\mu^{(M)}$ resulting from $M \in \N$ iterations \eqref{eq:iterationsFinal}. Then, the dynamics \eqref{eq:augmentedSys} are continuous and piecewise affine in~$\xb$.
\end{cor}

\begin{rem}
\label{rem:augmentedSystems}
In \cite{SchulzeDarup2019_ECC_ADMM}, the special case $M=1$ is considered.
The corresponding closed-loop dynamics are captured by an augmented system that is similar to but slightly different from \eqref{eq:augmentedSys}. To identify the connection between the systems in \eqref{eq:augmentedSys} and \cite[Eq.~(21)]{SchulzeDarup2019_ECC_ADMM}, we note that $\mu^{(1)}=\rho \Kb^{(1)} \xb - \rho z^{(1)}$ and consequently
\begin{equation}
\label{eq:mukPlus1M1}
\mu^{(0)}(k+1)=D_z \mu^{(1)}(k) = \rho D_z \Kb^{(1)} \xb(k) - \rho D_z z^{(1)}.
\end{equation}
for $M=1$. 
Relation \eqref{eq:mukPlus1M1} allows to formulate the augmented system \eqref{eq:augmentedSys} more compactly with an augmented control action that only depends on $z^{(1)}$. This modification leads to the augmented system in \cite[Eq.~(21)]{SchulzeDarup2019_ECC_ADMM}.
\end{rem}

\section{Linear dynamics around the augmented origin}
\label{sec:LinearAroundOrigin}

The augmented input $\vb ( \xb)$ results from $M$ ADMM iterations. In every iteration, evaluating $z^{(j+1)}$  includes a projection. Apart from the projection,
the iterations~\eqref{eq:iterationsFinal} are linear.
We next derive conditions under which the projections are effectless (in the sense that $\proj_\ZSet(z)=z$) and, hence, under which $\vb(\xb)$ is linear in $\xb$. 
As a preparation, we define the matrices 
\begin{equation}
\label{eq:Kj}
\Kb^{(j)}:=\begin{pmatrix} \left( \sum_{i=0}^{j-1} (\rho E_{11})^i \right) \!E_{12} F &\quad (\rho E_{11})^j &\quad  (\rho E_{11})^{j-1} (\frac{1}{\rho} I_q - E_{11})
 \end{pmatrix}
\end{equation}
for every $j \in \{1,\dots, M\}$, where we note that~\eqref{eq:Kj} for $j=1$ is in line with the definition of $\Kb^{(1)}$ in~\eqref{eq:K1}.
These matrices are instrumental for the following result.

\begin{prop}
\label{prop:zjLinearMujZero}
Let $M \in \N$ with $M\geq 1$ and let $\xb \in \R^r$ be such that
\begin{equation}
\label{eq:conditionKjInZ}
\Kb^{(j)} \xb \in \ZSet 
\end{equation}
for every $j \in \{1,\dots, M\}$. Then,
\begin{equation}
\label{eq:zjLinearMu0}
z^{(j)} = \Kb^{(j)} \xb  \qquad \text{and} \qquad \mu^{(j)} = 0_q
\end{equation}
for every $j \in \{1,\dots, M\}$.
\end{prop}

\begin{proof}
We prove the claim by induction. For $j=1$, we have  $z^{(1)}=\proj_{\ZSet}(\Kb^{(1)} \xb)$ and $\mu^{(1)}=\rho (\Kb^{(1)} \xb -  z^{(1)})$  according to~\eqref{eq:zjMujPlus1}. 
Now, \eqref{eq:conditionKjInZ} implies $z^{(1)}=\Kb^{(1)} \xb$ and consequently  $\mu^{(1)}=0_q$ as proposed. 
 We next show that \eqref{eq:zjLinearMu0} holds for $j+1$ if it holds for $j$.
To this end, we first note that satisfaction of \eqref{eq:zjLinearMu0} for $j$ implies
\begin{equation}
\label{eq:zjPlus1K1Kjx}
z^{(j+1)}= \proj_{\ZSet} \left(  \Kb^{(1)} \begin{pmatrix}
x \\
z^{(j)}\\
\mu^{(j)}
\end{pmatrix}  \right) = \proj_{\ZSet} \left(  \Kb^{(1)} \begin{pmatrix}
x \\
\Kb^{(j)} \xb \\
0_{q}
\end{pmatrix}  \right).
\end{equation}
By definition of $\Kb^{(j)}$, we further obtain
\begin{align}
\nonumber
\Kb^{(1)} \!\begin{pmatrix}
x \\
\Kb^{(j)} \xb \\
0_{q}
\end{pmatrix} &=  \begin{pmatrix}
E_{12} F &\,\,\,\, \rho E_{11} 
\end{pmatrix} \begin{pmatrix}
x \\
\Kb^{(j)} \xb
\end{pmatrix}  \\
\nonumber
&= \begin{pmatrix} \!E_{12} F \!+  \!\left( \sum_{i=1}^{j} (\rho E_{11})^{i} \right)\! E_{12} F  &\,\,\, (\rho E_{11})^{j+1} &\,\,\,  (\rho E_{11})^{j} (\frac{1}{\rho}I_q - E_{11})\! 
 \end{pmatrix} \xb \\
  \label{eq:KjPlus1x}
 &= \Kb^{(j+1)} \xb.
\end{align}
Substituting \eqref{eq:KjPlus1x} in \eqref{eq:zjPlus1K1Kjx} and using \eqref{eq:conditionKjInZ} (for $j+1$) implies
\begin{equation}
\label{eq:zjPlus1Linear}
z^{(j+1)} = \proj_{\ZSet} \left(  \Kb^{(j+1)} \xb  \right) = \Kb^{(j+1)} \xb
\end{equation}
as proposed. 
It remains to prove $\mu^{(j+1)} = 0_q$, which follows from
$$
\mu^{(j+1)} = \rho \left(  \Kb^{(1)} \begin{pmatrix}
x \\
z^{(j)} \\
\mu^{(j)}
\end{pmatrix} -z^{(j+1)} \right) = \rho \left(  \Kb^{(1)} \begin{pmatrix}
x \\
\Kb^{(j)} \xb\\
0_q
\end{pmatrix} -\Kb^{(j+1)} \xb \right) = 0_q,
$$
where we used \eqref{eq:zjLinearMu0} for $j$ and relations \eqref{eq:zjMujPlus1},  \eqref{eq:KjPlus1x}, and \eqref{eq:zjPlus1Linear}.
\end{proof}

Obviously, conditions~\eqref{eq:conditionKjInZ} are satisfied for $\xb = 0_r$. Not too surprisingly, the conditions are also satisfied in a neighborhood of the augmented origin and we will investigate this neighborhood in more detail in the next section.
For now, we study the effect of the conditions~\eqref{eq:conditionKjInZ} on the closed-loop dynamics. According to Proposition~\ref{prop:zjLinearMujZero}, having~\eqref{eq:conditionKjInZ} implies 
$$
\vb(\xb)=\begin{pmatrix}
\Kb^{(M)}  \\
0_{q \times r} 
\end{pmatrix} \xb
$$
for $M\geq 1$. Hence, around the augmented origin, the piecewise affine dynamics of the augmented system~\eqref{eq:augmentedSys} turn into the linear dynamics
\begin{equation}
\label{eq:linearDynamics}
\xb(k+1) = \Sb_M \xb(k),
\end{equation}
where
\begin{align}
\label{eq:SM}
\Sb_M &:=\Ab + \Bb \begin{pmatrix}
\Kb^{(M)}  \\
0_{q \times r} 
\end{pmatrix} \\
\nonumber
&= \!\begin{pmatrix}
A + B C_u \left( \sum_{i=0}^{M-1} (\rho E_{11})^i \right) \!E_{12} F & B C_u (\rho E_{11})^M   & B C_u  (\rho E_{11})^{M-1} (\frac{1}{\rho} I_q -  E_{11}) \\
D_z \left( \sum_{i=0}^{M-1} (\rho E_{11})^i \right) \!E_{12} F & D_z  (\rho E_{11})^M  &D_z  (\rho E_{11})^{M-1} (\frac{1}{\rho} I_q -  E_{11}) \\
0_{q \times n} &  0_{q\times q} &  0_{q\times q}
\end{pmatrix}\!.
\end{align}
Clearly, the eigenvalues of the matrix $\Sb_M$ determine whether the closed-loop behavior around the augmented origin is asymptotically stable or not. 
At this point, we observe that three categories of parameters have an effect on $\Sb_M$. First, the parameters of the original system in terms of $A$ and $B$. Second, the parameters $Q$, $R$, and $N$ of the MPC. Third, the parameters $\rho$,  $M$, and $D_z$ of the ADMM scheme. We note that $C_u$ is not counted as a parameter since
there seems to be no competitive alternative to the choice in \eqref{eq:uADMM}.
We further note that the update matrix $D_\mu$ has no effect on $\Sb_M$. 

Since we are dealing with a stabilizable pair $(A,B)$, an MPC parametrized as in Section~\ref{sec:Background} is stabilizing in some neighborhood around the origin for every prediction horizon $N\geq 1$. In particular, this neighborhood includes 
the set, where the MPC acts identical to the linear quadratic regulator (LQR), i.e., where the MPC law~\eqref{eq:optimalInput} is equivalent to
\begin{equation}
\label{eq:KLQR}
u(k)=K x(k) \qquad \text{with} \qquad K:=-(R + B^\top PB)^{-1} B^\top P A.
\end{equation}
Hence, instability can only result from inappropriate choices for $\rho$, $D_z$, and $M$. Based on this observation, it is obviously interesting to study the effect of the ADMM parameters on the eigenvalues of $\Sb_M$.
As a step in this direction, we next present two fundamental results: A general lower bound for the number of zero eigenvalues of $\Sb_M$   is given in Proposition~\ref{prop:eigenvaluesSM} and the existence of a stabilizing parameter set in terms of $\rho$, $D_z$, and $M$ is guaranteed by Proposition~\ref{prop:stableParameters}.

\begin{prop}
\label{prop:eigenvaluesSM}
Let $\rho \in \R$ with $\rho>0$,  let $M \in \N$ with $M \geq 1$, and let $D_z \in \R^{q \times q}$. Then, $\Sb_M$ as in~\eqref{eq:SM} has at least $q+p-n=(2 N-1) n + N m$ zero eigenvalues.
\end{prop}

We use the following lemma to prove Proposition~\ref{prop:eigenvaluesSM}.

\begin{lem}
\label{lem:rankE11M}
Let $\rho \in \R$ with $\rho>0$ and let $M \in \N$ with $M \geq 1$.
Then, 
\begin{equation}
\label{eq:rankE11Bound}
\rank(E_{11}^M )\leq \rank(E_{11}) = q - p = N m
\end{equation}
holds for $E_{11}$ as in~\eqref{eq:EBlocks}.
\end{lem}

\begin{proof}
Clearly, the first inequality in~\eqref{eq:rankE11Bound} represents a standard result and the last relation in~\eqref{eq:rankE11Bound} holds by definition of $q$ and $p$. Hence, it remains to prove  $\rank(E_{11}) = q - p$. To this end, we note that $E$ as in~\eqref{eq:EBlocks} results from inverting a $2 \times 2$ block matrix, where the upper-left block (i.e., $H+\rho I_q$) is invertible and where the off-diagonal blocks (i.e., $G^\top$ and $G$) have full rank as apparent from~\eqref{eq:definitionGF}. Thus, we obtain
\begin{equation}
\label{eq:E11}
E_{11} = (H+\rho I_q)^{-1} -   (H+\rho I_q)^{-1} G^\top  \left( G (H+\rho I_q)^{-1} G^\top \right)^{-1} \!G  (H+\rho I_q)^{-1}
\end{equation}
(see, e.g., \cite[Thm.~2.1]{Lu2002}).
Due to $\rank(H+\rho I_q)=q$ and $E_{11} \in \R^{q \times q}$, we next find
\begin{align}
\nonumber
\rank(E_{11})   &= \rank \left( (H+\rho I_q) E_{11} (H+\rho I_q) \right) \\
\label{eq:rankExpression}
&= \rank \left( H+\rho I_q -   G^\top  \left( G (H+\rho I_q)^{-1} G^\top \right)^{-1} \!G \right). 
\end{align}
To further investigate~\eqref{eq:rankExpression}, we choose the matrix $\Delta G \in \R^{(q-p) \times q}$ such that
$$
\ol{G}:=\begin{pmatrix}
G \\
\Delta G
\end{pmatrix}
$$
is invertible, where we note that such a choice is always possible. Using the latter matrix, $H+\rho I_q$ can be rewritten as
\begin{equation}
\label{eq:HrhoG}
 H+\rho I_q=\ol{G}^\top \ol{G}^{-\top} ( H+\rho I_q) \ol{G}^{-1} \ol{G} = \ol{G}^\top \left( \ol{G} ( H+\rho I_q)^{-1} \ol{G}^{\top} \right)^{-1} \ol{G}.
\end{equation}
After substituting \eqref{eq:HrhoG} in~\eqref{eq:rankExpression} and some lengthy but basic manipulations, we find
$$
\rank \left( H+\rho I_q -   G^\top  \left( G (H+\rho I_q)^{-1} G^\top \right)^{-1} \!G \right) = \rank \left( \Delta G \right).
$$ 
This completes the proof since $\rank \left( \Delta G \right)=q-p$.
\end{proof}

\begin{proof}[of Prop.~\ref{prop:eigenvaluesSM}]
To prove the claim, we first note that $\Sb_M$ has obviously at least $q$ zero eigenvalues since the last $q$ rows of $\Sb_M$ consist of zero entries. The remaining $q+n$ eigenvalues correspond to the eigenvalues of the upper-left submatrix
\begin{align*}
\ul{\Sb}_M &:= \begin{pmatrix}
A + B C_u \left( \sum_{i=0}^{M-1} (\rho E_{11})^i \right) \!E_{12} F &\,\,\,\,  B C_u (\rho E_{11})^M    \\
D_z \left( \sum_{i=0}^{M-1} (\rho E_{11})^i \right) \!E_{12} F &\,\,\,\,  D_z  (\rho E_{11})^M  
\end{pmatrix} \\
&=\begin{pmatrix}
A & 0_{n\times q}  \\
0_{q \times n} & 0_{q\times q}  
\end{pmatrix} + \begin{pmatrix}
B C_u \\
D_z 
\end{pmatrix} \begin{pmatrix}
\left( \sum_{i=0}^{M-1} (\rho E_{11})^i \right) \!E_{12} F &\quad (\rho E_{11})^M 
\end{pmatrix}.
\end{align*}
Using standard rank inequalities (see, e.g, \cite[p.~13]{Horn1985}) and the result from Lemma~\ref{lem:rankE11M}, we find
\begin{align*}
\rank(\ul{\Sb}_M) &\leq \rank(A) + \min \left\{ \rank \left( \!\begin{pmatrix}
B C_u \\
D_z 
\end{pmatrix}\!\right)\!,\, \rank \left( \begin{pmatrix}
\left( \sum_{i=0}^{M-1} \rho^i E_{11}^i \right) \!E_{12} F &\,\,\, \rho^M E_{11}^M 
\end{pmatrix} \right) \right\} \\
&\leq n + \min \left\{ q, n + \rank(E_{11}^M) \right\} \leq n + \min\{q,n+q-p \} = 2 n + Nm. 
\end{align*}
Hence, $\ul{\Sb}_M$ has at least $q+n-2n - Nm = (N-1)n$ zero eigenvalues, which implies that $\Sb_M$ has at least $(N-1)n+q=(2N-1)n+Nm$ zero eigenvalues. 
\end{proof}

\begin{prop}
\label{prop:stableParameters}
Let $M = 1$ and let $D_z \in \R^{q \times q}$. Then, there exists a $\rho>0$ such that $\Sb_M$ is Schur stable.
\end{prop}

\begin{proof}
Clearly, $\Sb_M$ is Schur stable if and only if $\ul{\Sb}_M$ is Schur stable. Now, for $M=1$, we have
\begin{equation}
\label{eq:S1}
\ul{\Sb}_1 = \begin{pmatrix}
A + B C_u E_{12} F &\,\,\,\, \rho B C_u E_{11}    \\
D_z E_{12} F &\,\,\,\, \rho D_z  E_{11}  
\end{pmatrix}.
\end{equation}
 It remains to show that there exists an $\rho>0$ such that $\ul{\Sb}_1$ is Schur stable, where we note that $E_{11}$ and $E_{12}$ depend on $\rho$ as apparent from~\eqref{eq:EBlocks}. It will turn out that a sufficiently small $\rho$ implies Schur stability of $\ul{\Sb}_1$. In this context, we study the limit $\lim_{\rho \rightarrow 0} E$ and find
 \begin{equation}
\label{eq:ERho0}
\begin{pmatrix}
E_{11}^\ast & E_{12}^\ast \\
(E_{12}^\ast)^\top & E_{22}^\ast
\end{pmatrix} := \lim_{\rho \rightarrow 0} \begin{pmatrix}
E_{11} & E_{12} \\
E_{12}^\top & E_{22}
\end{pmatrix} = \begin{pmatrix}
H &\,\,\, G^\top \\
G &\,\,\, 0_{p \times p} 
\end{pmatrix}^{-1}
\end{equation}
according to \eqref{eq:EBlocks}. Based on this result in combination with~\eqref{eq:S1}, we infer
\begin{equation}
\label{eq:limitS1Rho0}
\ul{\Sb}_1^\ast:=\lim_{\rho \rightarrow 0} \,\, \ul{\Sb}_1 =  \begin{pmatrix}
A + B C_u E_{12}^\ast F &\,\,\,\, 0_{n \times q}    \\
D_z E_{12}^\ast F &\,\,\,\, 0_{q \times q} 
\end{pmatrix}.
\end{equation}
As a consequence, for every $\delta >0$, there exists a $\rho>0$ such that $\| \ul{\Sb}_1 - \ul{\Sb}_1^\ast\| < \delta$ for some matrix norm $\|\cdot\|$. Moreover, since the spectrum of a  matrix depends continuously on its entries (see, e.g., \cite{Elsner1982}), for every $\epsilon>0$, there exists a $\delta >0$ such that $\| \ul{\Sb}_1 - \ul{\Sb}_1^\ast\| < \delta$ implies
\begin{equation}
\label{eq:eigenvaluesS1Epsilon}
\max_i \, \min_j \, |\lambda_i - \lambda_j^\ast | < \epsilon,
\end{equation}  
where $\lambda_i$ and $\lambda_j^\ast$ denote the $i$-th and $j$-th eigenvalue of $\ul{\Sb}_1$ and $\ul{\Sb}_1^\ast$, respectively.
In summary, for every $\epsilon>0$, there exists a $\rho>0$ such that~\eqref{eq:eigenvaluesS1Epsilon} holds.
We next show that $\ul{\Sb}_1^\ast$ is Schur stable. This completes the proof since \eqref{eq:eigenvaluesS1Epsilon} implies that, for a sufficiently small $\rho$, the eigenvalues of $\ul{\Sb}_1$ get arbitrarily close to those of $\ul{\Sb}_1^\ast$.
Now, it is apparent from \eqref{eq:limitS1Rho0} that $\ul{\Sb}_1^\ast$ is Schur stable if and only if $A + B C_u E_{12}^\ast F$ is Schur stable.
At this point, we note that the matrix on the right-hand side of~\eqref{eq:ERho0} is related to the unconstrained solution of the original QP~\eqref{eq:QP}. In fact, the optimizer of~\eqref{eq:QP} subject to $z\in \R^q$ instead of $z \in \ZSet$ satisfies 
\begin{equation}
\label{eq:unconstrainedZOpt}
\begin{pmatrix}
H &\,\,\, G^\top \\
G &\,\,\, 0_{p \times p} 
\end{pmatrix} \begin{pmatrix}
z^\ast \\
\ast
\end{pmatrix} = \begin{pmatrix}
0 \\
F x
\end{pmatrix}.
\end{equation}
The unconstrained solution of~\eqref{eq:QP}  is, on its own, related to the LQR. In fact, we have $C_u z^\ast = K x$ for $z^\ast$ as in~\eqref{eq:unconstrainedZOpt}. Taking \eqref{eq:ERho0} into account, we further obtain $z^\ast =E^\ast_{12} F x$. Hence,  $K=C_u E^\ast_{12} F$ and consequently
$A + B C_u E_{12}^\ast F = A + B K$. Clearly, $A + B K$ is Schur stable since the LQR stabilizes every stabilizable pair $(A,B)$.
\end{proof}

\section{Positive invariance around the augmented origin}
\label{sec:positiveInvariance}

In the previous section, we showed that the closed-loop behavior around the (augmented) origin obeys the linear dynamics \eqref{eq:linearDynamics}. In this section, we  analyze the neighborhood where \eqref{eq:linearDynamics} applies in more detail. The starting point for this analysis are the conditions \eqref{eq:conditionKjInZ} that imply linearity according to Proposition~\ref{prop:zjLinearMujZero}. As a consequence, the dynamics \eqref{eq:linearDynamics} apply to all states $\xb$ in the set
$$
\KSet_M := \left\{ \xb\in \R^r \,|\, \Kb^{(j)} \xb \in \ZSet, \, \forall j \in \{1,\dots,M\} \right\}.
$$
Unfortunately, having $\xb \in \KSet_M$ does not imply $\xb^+ :=\Sb_M \xb \in \KSet_M$. Hence, we next study the largest positively invariant set (for the linear dynamics) contained in $\KSet_M$. Clearly, this set corresponds to
 \begin{equation}
 \label{eq:setPM}
 \PSet_M:=\{ \xb\in \R^r \,|\, \Sb_M^k \xb \in \KSet_M , \, \forall k \in \N \}.
 \end{equation} 
Now, assume an augmented state trajectory enters $ \PSet_M$ at step $k^\ast$, i.e., $\xb(k^\ast) \in \PSet_M$. Then, all subsequent inputs, i.e.,
 \begin{equation}
 \label{eq:inputsLinear}
u(k)=C_u \Kb^{(M)} \Sb_M^{k-k^\ast} \xb(k^\ast)
 \end{equation}
for every $k\geq k^\ast$, satisfy the constraints $\USet$ by construction. In fact, for $k\geq k^\ast$, we have
$\xb(k)=\Sb_M^{k-k^\ast} \xb(k^\ast) \in \PSet_M$ due to $\xb(k^\ast) \in \PSet_M$ and consequently $\Kb^{(M)} \xb(k) \in \ZSet$ and $C_u \Kb^{(M)} \xb(k) \in \USet$.
However, the original states
 \begin{equation}
 \label{eq:statesLinear}
x(k) = C_x \Sb_M^{k-k^\ast} \xb(k^\ast) \qquad \text{with} \qquad C_x:=\begin{pmatrix}
I_n &\quad 0_{n \times (r-n)}
\end{pmatrix} 
 \end{equation}
 may or may not satisfy the constraints $\XSet$ for $k\geq k^\ast$. 
To compensate for this draw-back, we focus on positively invariant subsets of $
 \PSet_M$ that take the original state constraints explicitly into account.
In addition, it turns out to be useful to restrict our attention to sequences
$$
 z^{(0)}(k) = C_z \Sb_M^{k-k^\ast} \xb(k^\ast) \qquad \text{with} \qquad C_z:=\begin{pmatrix}
0_{q \times n} &\quad I_q &\quad 0_{q \times q}
\end{pmatrix} 
$$
 that satisfy the constraints $\ZSet$. 
Hence, we consider the set
 \begin{equation}
 \label{eq:subsetPM}
 \PSet_M^\ast:=\{ \xb\in \R^r \,|\, \Cb_M \Sb_M^k \xb \in \XSet \times \ZSet \times  \ZSet^M, \, \forall k \in \N \}
 \end{equation} 
with $\ZSet^{M}:=\underbrace{\ZSet \times \dots \times \ZSet}_{(M)-\text{times}}$ and
\vspace{-3mm}
 \begin{equation}
 \label{eq:CM}
\Cb_M := \begin{pmatrix}
C_x \\ C_z \\ \Kb^{(1)} \\ \vdots \\ \Kb^{(M)}.
\end{pmatrix},
 \end{equation}
We note that the first two blocks in $\Cb_M$ incorporate the conditions $C_x \Sb_M^k \xb \in \XSet$ and $C_z \Sb_M^k \xb \in \ZSet$, respectively, whereas the last $M$ blocks refer to $\Sb_M^k \xb \in \KSet_M$. Hence, we have $\PSet_M^\ast\subseteq \PSet_M$ by construction. 
Apparently, the sets \eqref{eq:setPM} and \eqref{eq:subsetPM} are similar to the output admissible sets studied in \cite{Gilbert1991}. According to \cite[Thms.~2.1 and 4.1]{Gilbert1991}, $ \ul{\PSet}_M$ as in~\eqref{eq:subsetPM} is bounded and finitely determined if (i) $\Sb_M$ is Schur stable, (ii) the pair $(\Cb_M,\Sb_M)$ is observable, (iii) $\XSet \times \ZSet \times \ZSet^M$ is bounded, and (iv) $0_{n+(M+1)q}$ is in the interior of $\XSet \times \ZSet \times \ZSet^M$.
In this context, we recall that finite determinedness implies the existence of a finite $\ol{k} \in \N$ such that
$$
  \PSet_M^\ast=\left\{ \xb \in \R^r \,|\, \Cb_M \Sb_M^k \xb \in \XSet \times \ZSet \times \ZSet^M , \, \forall k \in \{0,\dots,\ol{k}\} \right\}.
 $$
Now, $\XSet \times \ZSet \times \ZSet^M$ is bounded and contains the origin as an interior point by construction. We already analyzed Schur stability of $\Sb_M$ in the previous section and found that stabilizing parameters $\rho$, $D_z$, and $M$ always exist (see Prop.~\ref{prop:stableParameters}). Thus, it remains to study observability of $(\Cb_M,\Sb_M)$. 
To this end, we first derive the following lemma.

\begin{lem}
\label{lem:rhoIMinusE11}
Let $\rho \in \R$ with $\rho>0$. Then, $\frac{1}{\rho} I_q -  E_{11}$ is positive definite.
\end{lem}

\begin{proof}
Clearly, $H+\rho I_q$ is symmetric and positive definite. Hence, $\frac{1}{\rho} I_q -  E_{11}$ is positive definite if and only if the matrix  
\begin{align*}
&(H +\rho I_q) \left(\frac{1}{\rho}I_q -  E_{11} \right)\! (H+\rho I_q) \\
=\,\, & \frac{1}{\rho} (H+\rho I_q)^2\! -H\!-\rho I_q +   G^\top \! \left( G (H+\rho I_q)^{-1} G^\top \!\right)^{-1}\!G
\end{align*}
is positive definite, where the right-hand side of the equation results from~\eqref{eq:E11}. Now, positive definiteness of the latter matrix can be easily verified since it is the sum of the positive definite matrix
$$
\frac{1}{\rho} (H+\rho I_q)^2 -H-\rho I_q = (H+\rho I_q) \left( \frac{1}{\rho} (H+\rho I_q) - I_q \right) =  \frac{1}{\rho} H^2+H
$$
 and the positive semi-definite matrix
$G^\top  \left( G (H+\rho I_q)^{-1} G^\top \right)^{-1} \!G$.
\end{proof}

Based on Lemma~\eqref{lem:rhoIMinusE11}, it is straightforward to prove observability of $(\Cb_M,\Sb_M)$.

\begin{prop}
\label{prop:observability}
Let $\rho \in \R$ with $\rho>0$, let $M \in \N$ with $M\geq 1$, and let $\Sb_M$ and $\Cb_M$ be as in~\eqref{eq:SM} and \eqref{eq:CM}, respectively. Then, the pair $(\Cb_M,\Sb_M)$ is observable.
\end{prop} 

\begin{proof}
We prove the claim by showing that the observability matrix has full rank, i.e., rank $r$. To this end, we note that the $r\times r$-matrix
 \begin{equation}
 \label{eq:submatrixCM}
\begin{pmatrix}
C_x \\ C_z \\ \Kb^{(1)} 
\end{pmatrix} = \begin{pmatrix}
I_n &\quad 0_{n\times q} &\quad 0_{n\times q}\\ 
0_{q\times n} &\quad I_q &\quad 0_{q\times q} \\
E_{12} F &\quad \rho E_{11} &\quad \frac{1}{\rho} I_q - E_{11}
\end{pmatrix}
\end{equation}
is a submatrix of $\Cb_M$ for every $M\geq 1$. Hence, a sufficient condition for $(\Cb_M,\Sb_M)$ being observable is \eqref{eq:submatrixCM} having full rank. Now, verifying that \eqref{eq:submatrixCM} has rank $r=n+2q$ is easy  since the diagonal blocks of the block-triangular matrix have rank $n$, $q$, and $q$, where the latter holds according to Lemma~\ref{lem:rhoIMinusE11}.
\end{proof}

Since $(\Cb_M,\Sb_M)$ is observable according to Proposition~\ref{prop:observability}, $\PSet_M^\ast$ as in \eqref{eq:subsetPM} is bounded and finitely determined for every stabilizing choice of $\rho$,  $M$, and $D_z$ (with $M\geq 1$). It is interesting to note that boundedness does not hold for the superset $\PSet_M$  in \eqref{eq:setPM} as shown in the following remark.

\begin{rem}
In contrast to $\PSet_M^\ast$, the set $\PSet_M$ is not bounded  since we have    
\begin{equation}
\label{eq:xInPM}
\xb^\ast:= \begin{pmatrix}
0 \\
z \\
-\rho \left( \frac{1}{\rho} I_q -  E_{11} \right)^{-1}\! E_{11} z
\end{pmatrix} \in \PSet_M
\end{equation}
for every $z \in \R^q$, where the inverse exists according to Lemma~\ref{lem:rhoIMinusE11}. To retrace~\eqref{eq:xInPM}, we note that
$$
\Kb^{(j)} \xb^\ast = \rho^j E_{11}^j z - \rho^{j} E_{11}^{j-1} \left( \frac{1}{\rho} I_q -  E_{11} \right)  \left( \frac{1}{\rho} I_q -  E_{11} \right)^{-1}\! E_{11} z = 0_q
$$
for every $j \in \{1,\dots,M\}$. Hence, $\xb^\ast \in \KSet_M$. Moreover, we clearly have $\Ab \xb^\ast = 0_r$ and 
$$
\Bb \begin{pmatrix}
\Kb^{(M)}  \\
0_{q \times r} 
\end{pmatrix}\xb^\ast =  \Bb \cdot 0_{2q} = 0_r.
$$
As a consequence, we find $\Sb_M \xb^\ast = 0_r$ by definition of $\Sb_M$ in~\eqref{eq:SM}. In combination, we obtain $\Sb_M^k \xb^\ast \in \KSet_M$ for every $k\in \N$ and thus $\xb^\ast \in \PSet_M$.
\end{rem}

We finally note that the characterization of positively invariant sets can be simplified in exchange for slightly conservative results. In fact, regarding~\eqref{eq:SM}, it is easy to see that 
$$
 \mu^{(0)}(k) = C_{\mu} \Sb_M^{k-k^\ast} \xb(k^\ast) = 0_q \qquad \text{with} \qquad C_{\mu}:=\begin{pmatrix}
0_{q \times (n+q)} &\quad I_q 
\end{pmatrix} 
$$
for every $k > k^\ast$, every $M \geq 1$, and every $\xb(k^\ast) \in \R^{r}$. Hence, during the investigation of positive invariance,  we could assume $\mu=0$ and restrict our attention to the $x$--$z$--subspace.
However, for brevity and in order to incorporate the case $ C_{\mu}  \xb(k^\ast) \neq 0_q$, we do not detail this modification.

\section{Cost-to-go around the augmented origin}
\label{sec:costToGo}

Let us assume that a trajectory of the augmented system~\eqref{eq:augmentedSys} enters $\PSet_M^\ast$ at time step $k^\ast \in \N$, i.e., $\xb(k^\ast) \in \PSet_M^\ast$. Positive invariance of $\PSet_M^\ast$ combined with the linear dynamics~\eqref{eq:linearDynamics} then imply 
$$
\xb(k) = \Sb_M^{k-k^\ast} \xb(k^\ast) \in \PSet_M^\ast
$$
for all $k \geq k^\ast$. For stabilizing parameters $\rho$,  $M$, and $D_z$ (that imply Schur stability of $\Sb_M$), we additionally find
$$
\lim_{k \rightarrow \infty} \Sb_M^{k-k^\ast} \xb(k^\ast) = 0_r,
$$
and hence convergence to the origin. Thereby, the evolution of the original states obeys~\eqref{eq:statesLinear} 
and the applied inputs are~\eqref{eq:inputsLinear}
for $k \geq k^\ast$. 
According to the following proposition, the corresponding infinite-horizon cost is given by~\eqref{eq:costToGo}. 

\begin{prop}
\label{prop:costToGo}
Let the parameters $\rho>0$, $D_z$, and $M$ be such that $\Sb_M$ is Schur stable, let
\begin{equation}
\label{eq:QQ}
\Qb:=C_x^\top Q C_x + (\Kb^{(M)})^\top C_u^\top R C_u \Kb^{(M)}, 
\end{equation}
and assume that $\xb(k^\ast) \in \PSet_M^\ast$ for some $k^\ast \in \N$.
Then,
\begin{equation}
\label{eq:costToGo}
\sum_{k=k^\ast}^{\infty} \ell(x(k),u(k)) = \xb^\top\!(k^\ast) \Pb \xb(k^\ast),
\end{equation}
where $\Pb$ is the solution of the Lyapunov equation
\begin{equation}
\label{eq:dlyap}
\Pb = \Qb + \Sb_M^\top \Pb \Sb_M.
\end{equation}
\end{prop}

\begin{proof}
The definition of the stage cost in \eqref{eq:stageCost} combined with the input and state sequences~\eqref{eq:inputsLinear}--\eqref{eq:statesLinear} immediately lead to 
\begin{align*}
&\quad\sum_{k=k^\ast}^{\infty} \ell(x(k),u(k)) \\
&= \xb^\top\!(k^\ast) \!\left(\sum_{k=k^\ast}^{\infty} (\Sb_M^{k-k^\ast})^\top \!\left(C_x^\top\! Q \,C_x + (C_u \Kb^{(M)})^{\!\top} \! R C_u \Kb^{(M)}\right) \Sb_M^{k-k^\ast}  \right)\! \xb(k^\ast),\\
&= \xb^\top\!(k^\ast) \!\left(\sum_{\Delta k=0}^{\infty} (\Sb_M^{\Delta k})^{\top} \Qb \,\Sb_M^{\Delta k}  \!\right)\! \xb(k^\ast),
\end{align*}
where the second equation results from~\eqref{eq:QQ} and the substitution $\Delta k:= k - k^\ast$. It remains to prove that
\begin{equation}
\label{eq:sumPP}
\sum_{\Delta k=0}^{\infty} (\Sb_M^{\Delta k})^\top \Qb \, \Sb_M^{\Delta k} = \Pb.
\end{equation}
At this point, we first note that the sum in~\eqref{eq:sumPP} converges since $\Sb_M$ is Schur stable. Moreover, $\Qb$ as in~\eqref{eq:QQ} is obviously positive semi-definte. Using standard arguments, it is then straightforward to show that $\Pb$ can be inferred from~\eqref{eq:dlyap}.
\end{proof}

\begin{rem}
The matrix $\Qb$ in \eqref{eq:dlyap} is usually required to be positive definite in order to guarantee a unique and positive definite solution $\Pb$. The Lyapunov equation can, however, also be solved for positive semi-definite $\Qb$ yielding a positive semi-definite matrix $\Pb$. A suitable algorithm can, e.g., be found in \cite{Hammarling1991}.
\end{rem}

\section{Design parameters of the real-time ADMM}
\label{sec:designParameters}

The previous sections provide some insights on the closed-loop dynamics of ADMM-based MPC.
However, the identified model~\eqref{eq:augmentedSys} and invariant set~\eqref{eq:subsetPM} depend on various parameters that need to be specified.
Analogously to the analysis of $\Sb_M$ in Section~\ref{sec:LinearAroundOrigin}, we can distinguish three categories of involved parameters: The system parameters, the MPC parameters, and the ADMM parameters. Here, we focus on suitable choices for the ADMM parameters $\rho$, $M$, $D_z$, $D_\mu$, $z^{(0)}(0)$, and $\mu^{(0)}(0)$ that correspond to the weighting factor in \eqref{eq:copyObjective}, the number of iterations~\eqref{eq:iterationsFinal}, the update matrices in~\eqref{eq:zMuUpdates}, and the ``free'' initial values in~\eqref{eq:augmentedSys}, respectively. We stress, in this context, that the initial state $x(0)$ of the original system~\eqref{eq:linearSys} is, of course, not freely selectable.
Now, it is well-known that the performance of ADMM significantly varies with $\rho$. Optimal choices for $\rho$ are available for some specific setups (see, e.g., \cite{Ghadimi2015,Giselsson2017}). Unfortunately, these setups do not match with the ADMM scheme considered here. Hence, in our numerical experiments in the next section, we choose $\rho \in \{1,10,100\}$ in order to cover different magnitudes.
Regarding the number of ADMM iterations, we consider $M \in \{1,5,10\}$ in the numerical benchmark.  For $M=1$, we deliberately reproduce the results from \cite{SchulzeDarup2019_ECC_ADMM} that focus on a single ADMM iteration per time step. 

In order to make reasonable choices for $D_z$ and $D_\mu$, we have to better understand the role of these update matrices. To this end, let us first ignore the system dynamics~\eqref{eq:linearSys} and assume $x(k+1)=x(k)$.
In this case, we can choose $D_z$ and $D_\mu$ such that the real-time iterations become classical ADMM iterations solving~\eqref{eq:QP} for a fixed state $x$.  
In fact, for 
\begin{equation}
\label{eq:DzDmuCopy}
D_z=D_\mu=I_q,
\end{equation}
 we find $z^{(0)}(k+1)=z^{(M)}(k)$ and $\mu^{(0)}(k+1)=\mu^{(M)}(k)$, which reflects the classical setup in the sense that $z^{(j)}(k+1)=z^{(M+j)}(k)$ and $\mu^{(j)}(k+1)=\mu^{(M+j)}(k)$. However, in reality, \eqref{eq:linearSys} usually implies $x(k+1) \neq x(k)$. Hence, the choice \eqref{eq:DzDmuCopy} might, in general, not be useful.
To identify suitable alternatives, it is helpful to recall the definition of $z$. As apparent from~\eqref{eq:zDefinition}, $z$ contains predicted states and inputs for $N$ steps. Intuitively, after applying the first predicted input according to~\eqref{eq:uADMM},  it is reasonable to reuse the remaining $N-1$ steps as initial guesses for $z^{(0)}(k+1)$.
This classical idea is omnipresent in MPC and also exploited for the real-time scheme in \cite[Sect.~2.2]{Diehl2005_Single}. Applying the shifting requires to extend the shortened predictions by one terminal step. In this context, a simple choice is $\hat{u}(N-1)=0_m$ and consequently $\hat{x}(N)=A\hat{x}(N-1)$. Clearly, this choice leads to
\begin{equation}
\label{eq:CzShiftZero}
D_z = \begin{pmatrix}
0_{l\times (n+m)} & I_{l} & 0_{l\times n} \\
0_{n\times (n+m)} & 0_{n\times l} &  I_n \\
0_{m \times (n+m)} & 0_{m \times l} & 0_{m \times n}\\
0_{n \times (n+m)}  & 0_{n \times l} & A
\end{pmatrix},
\end{equation}
where $l:=q-2n-m$ is introduced for brevity. Another popular choice is $\hat{u}(N-1)=K \hat{x}(N-1)$, which implies $\hat{x}(N)=(A+BK)\,\hat{x}(N-1)$ and 
\begin{equation}
\label{eq:CzShiftLQR}
D_z = \begin{pmatrix}
0_{l\times (n+m)} & I_{l} & 0_{l\times n} \\
0_{n\times (n+m)} & 0_{n\times l} &  I_n \\
0_{m \times (n+m)} & 0_{m \times l} & K\\
0_{n \times (n+m)}  & 0_{n \times l} & S
\end{pmatrix},
\end{equation}
where $S:=A+BK$.
It remains to apply the shifting to the update of the Lagrange multipliers $\mu$. 
In contrast to the predictions for $z$, it is hard to reasonably extend the shortened predictions for $\mu$. Hence, we choose 
\begin{equation}
\label{eq:CmuShift}
D_{\mu}=\begin{pmatrix}
0_{(q-n-m) \times (n +m)} & I_{q-n-m} \\
0_{(n+m) \times (n +m)} & 0_{(n+m) \times (q-n-m)}
\end{pmatrix}
\end{equation}
as a counterpart for \eqref{eq:CzShiftZero} as well as \eqref{eq:CzShiftLQR}.
For both cases, this choice can be interpreted as the assumption that the added terminal step in $z$ is optimal.

The previously discussed choices for $\rho$, $M$, $D_z$, and $D_\mu$ determine the dynamics of the augmented system~\eqref{eq:augmentedSys}. 
The closed-loop trajectory additionally depends on on the initial state
$$
\xb_0 = \xb(0) = \begin{pmatrix}
x(0) \\
z^{(0)}(0) \\
\mu^{(0)}(0)
\end{pmatrix}= \begin{pmatrix}
x_0 \\
z_0 \\
\mu_0
\end{pmatrix}\!.
$$
As mentioned above, $z_0$ and $\mu_0$ can be freely chosen. We propose three different initial choices for $z_0$ that are, to some extent, related to the three discussed choices for $D_z$.  
A naive choice for $z_0$ is $0_p$. In fact, this choice is only optimal if $x_0=0_n$, i.e., if the system is initialized at the setpoint.
A more reasonable initialization results for the predicted inputs $\hat{u}(0)=\dots=\hat{u}(N-1)=0_m$ and the related states $\hat{x}(k)=A^k x_0$. For this choice, $z_0$ can be written as $z_0 := D_0 x_0$ with
\begin{equation}
\label{eq:C0zero}
D_0^\top := \begin{pmatrix}\begin{pmatrix}
0_{m \times n} \\
A
\end{pmatrix}^\top & \dots & \begin{pmatrix}
0_{m \times n} \\
A^N
\end{pmatrix}^\top 
\end{pmatrix}.
\end{equation}
We note that this initial choice satisfies the input constraints by construction. However, the state constraints might be violated especially for unstable system matrices $A$.
In this case, the choice
\vspace{-3mm}
 \begin{equation}
\label{eq:C0LQR}
D_0^\top := \begin{pmatrix}\begin{pmatrix}
K S^0 \\
S
\end{pmatrix}^\top & \dots & \begin{pmatrix}
K S^{N-1} \\
S^N
\end{pmatrix}^\top 
\end{pmatrix}
\end{equation}
might be useful, which is based on the input predictions $\hat{u}(k)=K \hat{x}(k)$ and the related states $\hat{x}(k)=S^k x_0$. It is hard to construct initializations for the Lagrange multipliers that reflect the different approaches for $z_0$. We thus choose $\mu_0=0_p$ independent of the initialization for~$z_0$.

\begin{table}[htp]
\caption{Overview of different choices for the weighting factor $\rho$, the number of iterations $M$, the update matrices $D_z$ and $D_\mu$, and the initializations $z_0 = D_0 x_0$ (from left to right).}
\label{tab:rhoUpdatesInit}
\begin{center}
\small
\begin{tabular}{lc}
\toprule
 weighting\!\! & $\rho$ \\
\midrule
 small  & $1$ \\
 medium & $10$ \\
 large & $100$ \\
\bottomrule 
\end{tabular}\hspace{5mm}
\begin{tabular}{lc}
\toprule
 iterations & $M$ \\
\midrule
one  & $1$ \\
 five & $5$ \\
 ten & $10$ \\
\bottomrule  
\end{tabular}\hspace{5mm}
\begin{tabular}{lcc}
\toprule
 updates & $D_z$ & $D_\mu$ \\
\midrule
copy  & $I_q$ & $I_q$ \\
shift-zero & \eqref{eq:CzShiftZero} & \eqref{eq:CmuShift} \\
 shift-LQR & \eqref{eq:CzShiftLQR} & \eqref{eq:CmuShift}\\
\bottomrule 
\end{tabular}\hspace{5mm}
\begin{tabular}{lc}
\toprule
 initialization\!\! & $D_0$ \\
\midrule
 naive  & $0_{q \times n}$ \\
 zero & \eqref{eq:C0zero}  \\
 LQR & \eqref{eq:C0LQR} \\
\bottomrule 
\end{tabular}
\end{center}
\end{table}

A summary of the different choices for the parameter $\rho$, the number of iterations $M$, the update matrices $D_z$ and $D_\mu$, and the initialization $z_0$ is listed in Table~\ref{tab:rhoUpdatesInit}. Hence, by considering all combinations of the parameter choices, we obtain $81 =  3^4$
different realizations of the proposed real-time ADMM.

\section{Numerical benchmark}
\label{sec:benchmark}

To investigate the performance of the ADMM-based MPC, we apply the $81$ 
ADMM parametrizations from the previous section to the (discretized) double integrator 
with the system matrices
$$
A=\begin{pmatrix}1&1\\0&1\end{pmatrix} \quad \text{and} \quad  B=\begin{pmatrix}0.5\\1 \end{pmatrix}
$$
and the state and input bounds
$$
\ol{x}=-\ul{x}=\begin{pmatrix}
25 \\
5
\end{pmatrix} \quad \text{and} \quad  \ol{u}=-\ul{u}=1.
$$
The MPC cost functions in \eqref{eq:stageCost} are specified by $Q=I_2$, $R=0.1$, and $P$ as in~\eqref{eq:DARE}. The prediction horizon is chosen as $N=5$.

For every parametrization, we first investigate the linear dynamics~\eqref{eq:linearDynamics} around the augmented origin. More precisely,  we study the Schur stability of the augmented system matrix $\Sb_M$.
As pointed out in Section~\ref{sec:LinearAroundOrigin}, only the parameters $\rho$, $M$, and $D_z$ have an effect on $\Sb_M$. For the corresponding $3^3 = 27$ 
parametrizations, it is easy to verify numerically that $\Sb_M$ as in~\eqref{eq:SM} is always Schur stable.
This observation is promising since Schur stability has, so far, not been proven for specific parameter choices. In fact, Proposition~\ref{prop:stableParameters} merely proves the existence of stabilizing parameters and the proof focuses on small weighting factors $\rho$.
Moreover, one can easily verify that the (meaningless but possible) choice $\rho=10$, $M=1$, $D_z=-2 I_q$ results in an unstable matrix $\Sb_M$.
Now, since the pair $(\Cb_M,\Sb_M)$ (with $\Cb_M$ as in \eqref{eq:CM}) is always observable according to Proposition~\ref{prop:observability}, the set $\PSet_M^\ast$ (as in \eqref{eq:subsetPM}) is finitetely determined for all considered parametrizations.
We next compute $\PSet_M^\ast$ for every parametrization using the standard procedure in \cite[Sect.~III]{Gilbert1991}.
 The resulting sets $\PSet_M^\ast$ are high-dimensional. In fact, $\PSet_M^\ast$ is $r$-dimensional with 
 $$
 r=n+2q= n+2(n+m) N=62
 $$ 
 for $n=2$, $m=1$, and $N=10$ as in the example.
In order to get a feeling for the size of the various sets $\PSet_M^\ast$, we will analyze low-dimensional slices of the form
\begin{equation}
\label{eq:sliceS}
\SSet_{z,\mu} := \left\{ x \in \R^n\,\left|\, \xb=\begin{pmatrix}
x^\top & z^\top & \mu^\top
\end{pmatrix}^\top \in \PSet_M^\ast \right.\right\}
\end{equation}
that result from fixing $z$ and $\mu$ to some specific values.
It is easy to see that the slices $\SSet_{z,\mu}$ are subsets of the state constraints $\XSet$ for every choice of $z$ and $\mu$.
Moreover, the slices $\SSet_{z,\mu}$ have some similarities with the set  
\begin{equation}
\label{eq:setT}
\TSet:=\{ x \in \R^n \,|\, S^k x \in \DSet , \, \forall k \in \N \},  
\end{equation}
where $\DSet :=\{ x \in \XSet \,|\, Kx \in \USet \}$, where $S=A+BK$ as above, and where $K$ is as in~\eqref{eq:KLQR}. Clearly, the set $\TSet$ refers to the largest set, where the LQR can be applied without violating the constraints and where \eqref{eq:KLQR} applies (for $P$ as in \eqref{eq:DARE}). It can be easily computed (see Fig.~\ref{fig:SetsTrajsXs}) and typically serves as a terminal set for MPC. The similarities between $\SSet_{z,\mu}$ and $\TSet$ are as follows: For every state $x$ in these sets, the upcoming trajectories 
(resulting from ADMM-based MPC respectively classical MPC) are captured by linear (augmented) dynamics and converging to the (augmented) origin.
Hence, both sets provide a numerically relatively cheap underestimation of the domain of attraction (DoA) for the corresponding predictive control scheme. We compare the size of these underestimations by evaluating the ratio
\begin{equation}
\label{eq:ratioST}
\frac{\vol(\SSet_{z,\mu})}{\vol(\TSet)}
\end{equation}
for different $z$ and $\mu$. More precisely, we consider $z=D_0 x $ for the three variants of $D_0$ in Table~\ref{tab:rhoUpdatesInit} and $\mu=0_q$, i.e., we consider slices related to the different initializations for $z_0$ and $\mu_0$.
Numerical values for the ratios~\eqref{eq:ratioST} and all $81$
 parametrizations are listed in Table~\ref{tab:Benchmark} (see columns ``vol.'').
Interestingly, all values are larger than or equal to $1$. Hence, the DoA of the system controlled by the ADMM-based MPC is at least as large as the set $\TSet$. Moreover, for some parametrizations, $\SSet_{z,\mu}$ is significantly larger than $\TSet$. For example, the ratio~\eqref{eq:ratioST} evaluates to $31.00$ 
for the parametrization in line $13$ of Table~\ref{tab:Benchmark} and $M=5$
(see Fig.~\ref{fig:SetsTrajsXs}).  Another interesting observation is that the ratios are decreasing with $M$  for most parametrizations except for those in lines $4$, $5$, and $13$ of Table~\ref{tab:Benchmark}. One explanation for this trend is the fact that the number of rows in $\Cb_M$ (see \eqref{eq:CM}) and, consequently, the number of potential hyperplanes restricting $\PSet_M^\ast$ is increasing with $M$.

The slices $\SSet_{z,\mu}$ provide a useful underestimation of the DoA. However, in order to get a more complete impression of the DoA, we have to take into account initial states $x_0$ outside of $\SSet_{z,\mu}$. To this end, we randomly generated $500$ initial states $x_0$ that are feasible for the classical MPC, i.e., that are contained in 
$\FSet_{5}$ defined as in \eqref{eq:setFN}  (see Fig.~\ref{fig:SetsTrajsXs}). Although not enforced by a terminal set, all of these states are steered to the origin by the classical MPC. In fact, the corresponding trajectories enter the set $\TSet$ after at most $15$ time steps.
In contrast, not all trajectories resulting from the proposed ADMM-based MPC are converging.
The percentages of converging trajectories for the different parametrizations are listed in Table~\ref{tab:Benchmark} (see columns ``cnvg.''). In this context, a trajectory is counted as ``converging'' if it reaches the set $\PSet_M^\ast$ after at most $50$ time steps. 
Apparently, the numbers of converging trajectories differ significantly for the various parametrizations (see also Fig.~\ref{fig:SetsTrajsXs}). 
In fact, only $13\%$ of the trajectories converge for the parametrization in line $21$ in Table~\ref{tab:Benchmark} and $M=1$, whereas $100\%$ converge, e.g, for the parametrization in line $25$ and $M=10$.
We further observe that the convergence ratios are greater than $87\%$ whenever shifted updates are used (i.e., $D_z$ as in \eqref{eq:CzShiftZero} or \eqref{eq:CzShiftLQR}  and $D_\mu$ as in \eqref{eq:CmuShift}). In contrast, for copied updates (i.e., $D_z=D_\mu=I_q$), the majority of convergence rates is below $67\%$. A simple explanation for this observation are varying states $x(k)$.

\begin{figure}[tp]
\includegraphics[trim=1.2cm 10.2cm 1.2cm 10.2cm, clip=true,width=0.98\linewidth]{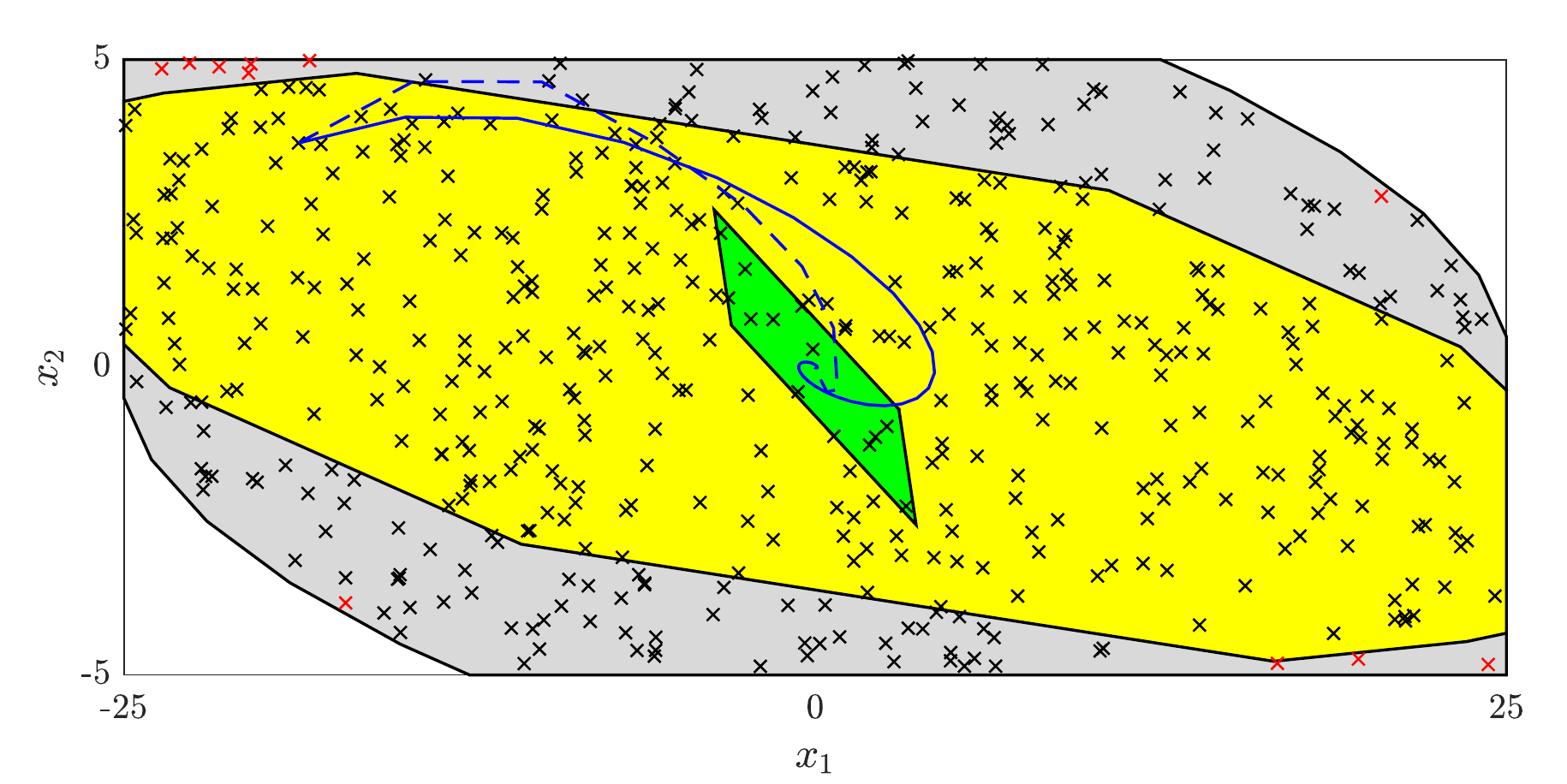}
\includegraphics[trim=1.2cm 10.2cm 1.2cm 10.2cm, clip=true,width=0.98\linewidth]{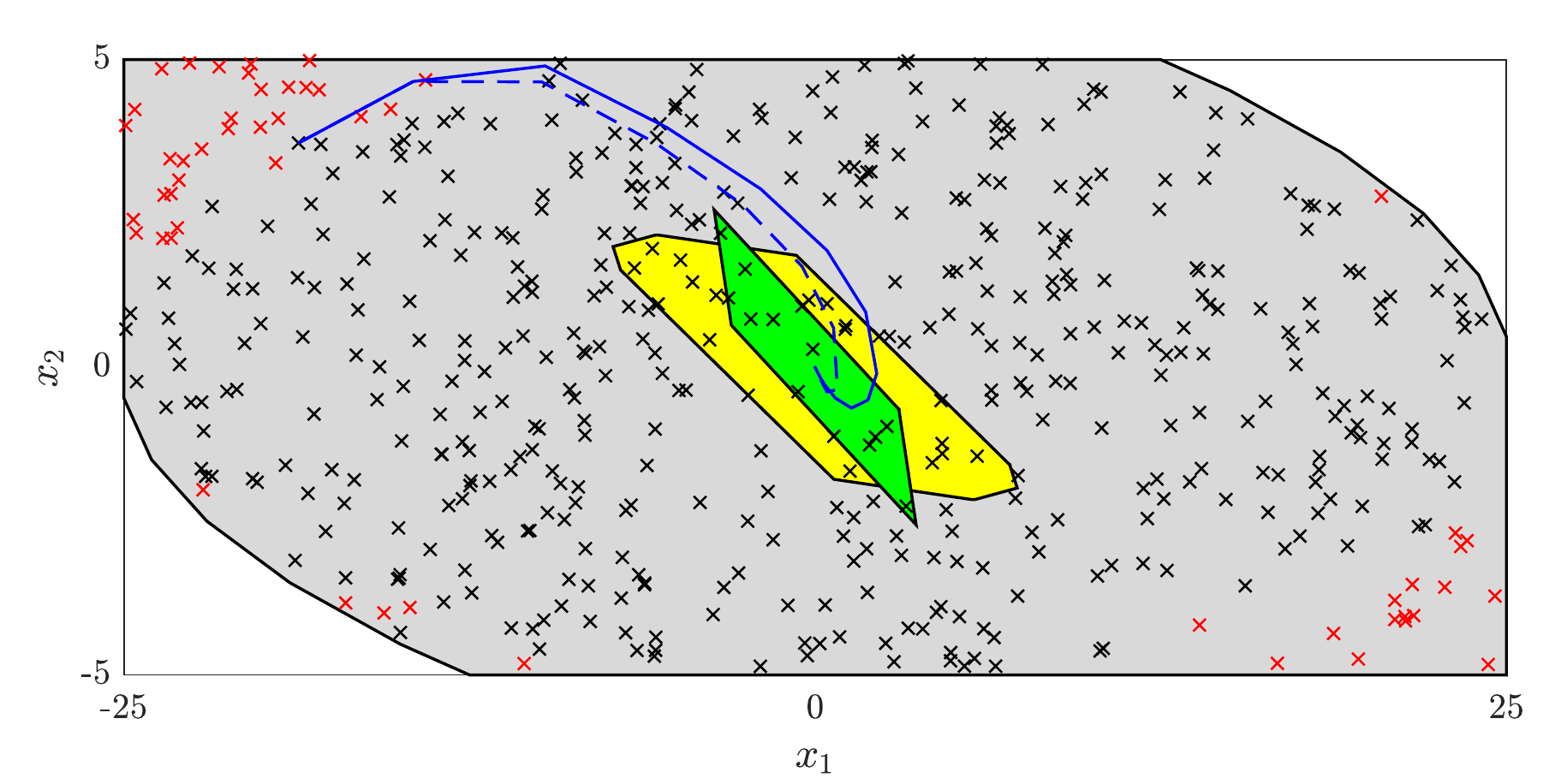}
\caption{Illustration of various sets, initial states, and trajectories resulting from the numerical example for $M=5$ and the parametrizations in lines 13 (top) and 23 (bottom) of Table~\ref{tab:Benchmark}, respectively. In both plots, the sets $\FSet_5$ (gray), $\SSet_{z_0,\mu_0}$ (yellow), and $\TSet$ (green) are shown. 
Moreover, the 500 generated initial states $x_0$ are marked with crosses. In this context, black, respectively red, crosses indicate whether the corresponding ADMM-based trajectories converge or not. Finally, the trajectories resulting from the ADMM-based MPC (solid blue) and the classical MPC (dashed blue) are depicted for the initial state $x_0^\top = (-18.680 \,\,\, 3.646)$.}
\label{fig:SetsTrajsXs}
\end{figure}

\newpage
It remains to investigate the performance of the ADMM-based MPC. To this end, we compare the overall costs of converging trajectories to the corresponding costs resulting from classical MPC. By definition, a converging trajectory enters $\PSet_M^\ast$ after at most $50$ time steps. Hence, according to Proposition~\ref{prop:costToGo}, the overall cost for a converging trajectory evaluates to
\begin{equation}
\label{eq:infCostADMM}
V_\infty^{\text{ADMM}}(x_0):= \sum_{k=0}^{\infty} \ell(x(k),u(k)) = \xb^\top\!(k^\ast) \Pb \xb(k^\ast) + \sum_{k=0}^{k^\ast-1} \ell(x(k),u(k)), 
\end{equation}
where $k^\ast\leq 50$ is such that $\xb(k^\ast) \in \PSet_M^\ast$. Similarly, the overall cost of a classical MPC trajectory can be calculated as
\begin{equation}
\label{eq:infCostMPC}
V_\infty^{\text{MPC}}(x_0):= \sum_{k=0}^{\infty} \ell(x(k),u(k))= \varphi(x(k_\infty)) + \sum_{k=0}^{k_\infty-1} \ell(x(k),u(k)), 
\end{equation}
where $k_\infty$ is such that $x(k_\infty) \in \TSet$.
We separately compute the performance ratio
$$
\frac{V_\infty^{\text{MPC}}(x_0)}{V_\infty^{\text{ADMM}}(x_0)}
$$
 for every converging trajectory and list the mean values for the different parametrizations in Table~\ref{tab:Benchmark} (see columns ``perf.'').
The (average) performance ratios indicate that the ADMM-based MPC results, in most cases, in a performance decrease. Nevertheless, the performance ratios are larger than $74\%$ for all cases with $\rho \in \{1,10\}$.
Furthermore, the performance ratios are increasing with $M$ for most parametrizations.
This observation is reasonable since the ADMM iterations~\eqref{eq:iterationsFinal} monotonically converge to the optimizers $z^\ast$ and $\mu^\ast$ in the sense that the  quantity
\begin{equation}
\label{eq:decreaseADMM}
\rho \| z^{(j)} - z^\ast \|_2^2 + \frac{1}{\rho} \| \mu^{(j)} - \mu^\ast \|_2^2
\end{equation}
 decreases with $j$ (see, e.g., \cite[App.~A]{Boyd2011}).
Nevertheless, being closer to the optimum does not necessarily imply constraint satisfaction. This is apparent from the convergence ratios in Table~\ref{tab:Benchmark} that are decreasing with $M$ in some cases.

In order to compare the proposed real-time iteration schemes with standard ADMM, we finally investigate the number of
iterations \eqref{eq:iterationsFinal} necessary to solve 
the arising QPs to a certain level of accuracy. More precisely, we 
evaluate the first iteration $j$ that satisfies 
\begin{equation}
\label{eq:zjAccurate}
\| z^{(j)} - z^\ast \|_2^2\leq 10^{-4}
\end{equation}
 and denote it with $M^\ast$ for each QP. Regarding the parametrization, we consider again the 27 variants in Table~\ref{tab:Benchmark} and consequently warmstarts similar to~\eqref{eq:zMuUpdates}. Since $z^{(M^\ast)} \approx z^\ast$, the warmstarts are only meaningful if we follow the trajectories resulting from the original MPC (and not those resulting from the real-time schemes). Hence, the listed $M^\ast$ in Table~\ref{tab:Benchmark} refer to the number of iterations required to obtain~\eqref{eq:zjAccurate} averaged over all QPs along all $500$ MPC trajectories for each parametrization. Apparently, $M^\ast$ is, on average, significantly larger than the considered number of iterations $M$ for the investigated real-time schemes. In fact, we observe $M^\ast \in [30.5,431.9]$ while $M \in [1,10]$ has been considered for the proposed schemes. 

\begin{landscape}
\begin{table*}[htp]
\centering
\caption{Numerical benchmark for the $81$ 
real-time ADMM parametrizations applied to the double integrator example. The abbreviations ``vol.'', ``cnvg.'', and ``perf.'' are short for volume, convergence, and performance ratio, respectively. The listed iterations $M^\ast$ are required for standard ADMM to achieve the accuracy~\eqref{eq:zjAccurate}.}
\label{tab:Benchmark}
\small
\begin{tabular}{ccccrrrrrrrrrr}
\toprule
\multicolumn{4}{c}{ADMM parameters} & \multicolumn{3}{c}{$M=1$}   & \multicolumn{3}{c}{$M=5$}  & \multicolumn{3}{c}{$M=10$} & $M^\ast$\, \\
\cmidrule(lr){1-4}\cmidrule(lr){5-7}\cmidrule(lr){8-10}\cmidrule(lr){11-13}
line & updates  & init. & $\rho$  & \multicolumn{1}{c}{\,\,\,vol.}&\multicolumn{1}{c}{\,\,cnvg.}&\multicolumn{1}{c}{\,\,perf.}&\multicolumn{1}{c}{\,\,\,vol.}&\multicolumn{1}{c}{\,\,cnvg.}&\multicolumn{1}{c}{\,\,perf.}&\multicolumn{1}{c}{\,\,\,vol.}&\multicolumn{1}{c}{\,\,cnvg.}&\multicolumn{1}{c}{\,\,perf.} &  \\
\midrule
1 &shift-LQR &LQR &$100$& $1.00$& $0.94$& $0.94$& $1.00$& $0.98$& $0.98$& $1.00$& $0.97$& $0.97$& $61.7$ \\
2 &shift-LQR &LQR &$10$& $1.00$& $0.94$& $0.93$& $1.00$& $0.93$& $1.00$& $1.00$& $0.97$& $1.00$& $30.5$ \\
3 &shift-LQR &LQR &$1$& $1.00$& $0.94$& $0.91$& $1.00$& $0.91$& $0.99$& $1.00$& $0.89$& $1.00$& $277.6$ \\
4 &shift-LQR &zero &$100$& $1.05$& $0.95$& $0.50$& $1.32$& $0.98$& $0.69$& $1.78$& $0.97$& $0.84$& $77.0$ \\
5 &shift-LQR &zero &$10$& $1.69$& $0.95$& $0.81$& $2.94$& $0.91$& $0.98$& $1.70$& $0.97$& $1.00$& $32.3$ \\
6 &shift-LQR &zero &$1$& $2.17$& $0.92$& $0.94$& $1.04$& $0.91$& $0.99$& $1.00$& $0.89$& $1.00$& $277.8$ \\
7 &shift-LQR &naive &$100$& $2.00$& $0.94$& $0.97$& $1.93$& $0.98$& $0.99$& $1.85$& $0.96$& $0.98$& $64.8$ \\
8 &shift-LQR &naive &$10$& $1.86$& $0.94$& $0.96$& $1.48$& $0.94$& $1.00$& $1.26$& $0.97$& $1.00$& $30.8$ \\
9 &shift-LQR &naive &$1$& $1.36$& $0.94$& $0.93$& $1.02$& $0.91$& $0.99$& $1.00$& $0.89$& $1.00$& $277.6$ \\
10 &shift-zero &LQR &$100$& $1.00$& $0.91$& $0.61$& $1.00$& $0.97$& $0.85$& $1.00$& $0.98$& $0.96$& $177.1$ \\
11 &shift-zero &LQR &$10$& $1.00$& $0.94$& $0.94$& $1.00$& $0.93$& $1.00$& $1.00$& $0.95$& $1.00$& $44.0$ \\
12 &shift-zero &LQR &$1$& $1.00$& $0.94$& $0.91$& $1.00$& $0.91$& $0.99$& $1.00$& $0.89$& $1.00$& $279.0$ \\
13 &shift-zero &zero &$100$& $26.36$& $0.97$& $0.14$& \,$31.00$& $0.98$& $0.60$& $17.07$& $0.95$& $0.82$& $192.4$ \\
14 &shift-zero &zero &$10$& $19.02$& $0.96$& $0.79$& $2.95$& $0.92$& $0.98$& $1.70$& $0.95$& $1.00$& $45.7$ \\
15 &shift-zero &zero &$1$& $2.26$& $0.92$& $0.94$& $1.04$& $0.91$& $0.99$& $1.00$& $0.89$& $1.00$& $279.2$ \\
16 &shift-zero &naive &$100$& $2.00$& $0.87$& $0.54$& $1.93$& $0.97$& $0.86$& $1.85$& $0.96$& $0.97$& $180.3$ \\
17 &shift-zero &naive &$10$& $1.86$& $0.94$& $0.97$& $1.48$& $0.95$& $1.00$& $1.26$& $0.95$& $1.00$& $44.3$ \\
18 &shift-zero &naive &$1$& $1.36$& $0.94$& $0.93$& $1.02$& $0.91$& $0.99$& $1.00$& $0.89$& $1.00$& $279.1$ \\
19 &copy &LQR &$100$& $1.00$& $0.77$& $0.72$& $1.00$& $0.98$& $0.67$& $1.00$& $1.00$& $0.81$& $319.7$ \\
20 &copy &LQR &$10$& $1.00$& $0.59$& $0.74$& $1.00$& $0.90$& $0.98$& $1.00$& $0.92$& $1.00$& $64.8$ \\
21 &copy &LQR &$1$& $1.00$& $0.13$& $0.83$& $1.00$& $0.52$& $0.87$& $1.00$& $0.65$& $0.95$& $431.8$ \\
22 &copy &zero &$100$& $12.95$& $0.66$& $0.24$& $11.82$& $0.98$& $0.56$& $8.58$& $0.98$& $0.78$& $335.2$ \\
23 &copy &zero &$10$& $11.16$& $0.95$& $0.78$& $2.61$& $0.89$& $0.98$& $1.70$& $0.93$& $1.00$& $66.6$ \\
24 &copy &zero &$1$& $2.11$& $0.17$& $0.90$& $1.05$& $0.53$& $0.87$& $1.00$& $0.64$& $0.95$& $431.9$ \\
25 &copy &naive &$100$& $2.00$& $0.81$& $0.82$& $1.93$& $0.98$& $0.68$& $1.86$& $1.00$& $0.82$& $322.9$ \\
26 &copy &naive &$10$& $1.87$& $0.64$& $0.83$& $1.50$& $0.90$& $0.99$& $1.28$& $0.93$& $1.00$& $65.1$ \\
27 &copy &naive &$1$& $1.38$& $0.15$& $0.87$& $1.03$& $0.52$& $0.87$& $1.00$& $0.65$& $0.95$& \,\,\,$431.8$ \\
\bottomrule
\end{tabular}
\end{table*}
\end{landscape}

\section{Conclusions and outlook}
\label{sec:Conclusions}

This paper focused on MPC based on real-time ADMM. The restriction to a finite number of ADMM iterations per time step allowed us to systematically analyze the dynamics of the controlled system. The first part of the analysis showed that the closed-loop dynamics can be described based on a nonlinear augmented model. The associated augmented state consists of the original states $x$, the decision variables~$z$, and the Lagrange multipliers  $\mu$.
The second part of the analysis revealed that the nonlinear dynamics turn into linear ones around the augmented origin. 
We further investigated where the linear dynamics apply and characterized a positively invariant set in the augmented state space. 
The third part of the analysis addressed the influence of  various ADMM parameters.
We motivated different choices for every parameter and evaluated their efficiency in a comprehensive numerical benchmark.
The benchmark clearly indicates that real-time ADMM is competitive to classical MPC for suitable parametrizations.
In fact, for some parametrizations (e.g., in line $8$ of Tab.~\ref{tab:Benchmark} with $M=10$), we found a high convergence ratio of $97\%$ and simultaneously a performance ratio of almost $100\%$, i.e., nearly optimal.
 
The obtained results extend the findings in \cite{SchulzeDarup2019_ECC_ADMM} in many directions. First,  the analysis in \cite{SchulzeDarup2019_ECC_ADMM} is restricted to one ADMM iteration per time step (i.e., $M=1$), whereas the new results support $M \geq 1$. 
 Second, Propositions~\ref{prop:eigenvaluesSM} and \ref{prop:stableParameters} provide novel findings on  Schur stability and eigenvalues of $\Sb_M$. 
Third, observability of the pair  $(\Cb_M,\Sb_M)$ has been formally proven in Proposition~\ref{prop:observability}. 

 While the presented results are more complete than the pioneering work \cite{SchulzeDarup2019_ECC_ADMM}, many promising extensions are left for future research.
First, the ADMM-based MPC here and in \cite{SchulzeDarup2019_ECC_ADMM} builds on  the ``uncondensed'' QP~\eqref{eq:QP}. It would be interesting to study ADMM-based MPC derived from the ``condensed'' QP in \cite{Ghadimi2015}. Second, the stability analysis of the augmented system is still incomplete. 
It may, however, be possible to extend the guaranteed domain of attraction beyond the linear regime by exploiting~\eqref{eq:decreaseADMM} or the  contraction estimates recently proposed in \cite{Zanelli2019}.
 Third, robustness against disturbances has not been considered yet. Fourth, real-world applications of the proposed predictive control scheme should be addressed.
In this context, it is interesting to note that real-time optimization schemes not only support embedded and fast controller implementations. In fact, they also pave the path for encrypted predictive control as in~\cite{SchulzeDarup2018_NMPC}.

\end{document}